\documentclass{amsart}
\usepackage{latexsym}
\usepackage{amsmath}
\usepackage{amssymb}
\usepackage{amsthm}
\usepackage{amscd}
\usepackage{enumerate} 
\usepackage{amssymb} 
\usepackage{mathrsfs}
\usepackage[all]{xy}
\usepackage{color}
\usepackage{graphicx}
\usepackage{mathtools}
\usepackage{comment}
\usepackage{multirow}
\usepackage{hyperref}
\hypersetup{colorlinks=true}

\title{Bogomolov-Sommese vanishing and liftability for surface pairs in positive characteristic}
\author{Tatsuro Kawakami}
\email{tatsurokawakami0@gmail.com}
\address{Department of Mathematics, Graduate School of Science, Kyoto University, Kyoto 606-8502, Japan}

\def\phi{\varphi}
\def\epsilon{\varepsilon}
\def\tilde{\widetilde}

\def\mapsto{\longmapsto}

\def\log{\operatorname{log}}
\def\Hom{\operatorname{Hom}}
\def\Spec{\operatorname{Spec}}

\def\Supp{\operatorname{Supp}}

\def\Im{\operatorname{Im}}

\def\Exc{\operatorname{Exc}}

\def\max{\operatorname{max}}

\def\Sym{\operatorname{Sym}}

\def\res{\operatorname{res}}

\newcommand{\Q}{\mathbb{Q}} 
\newcommand{\C}{\mathbb{C}} 
\newcommand{\R}{\mathbb{R}} 
\newcommand{\Z}{\mathbb{Z}}
\newcommand{\PP}{\mathbb{P}}

\newcommand{\sO}{\mathcal{O}}

\theoremstyle{plain}
\newtheorem{thm}{Theorem}[section] 

\newtheorem{prop}[thm]{Proposition}

\newtheorem{lem}[thm]{Lemma}
\theoremstyle{definition} 
\newtheorem{defn}[thm]{Definition}

\newtheorem{eg}[thm]{Example} 

\theoremstyle{remark}
\newtheorem{rem}[thm]{Remark}

\newtheorem{defn and notation}[thm]{Definition and Notation}
\newtheorem*{notation}{Notation} 
\newtheorem*{cl}{Claim}
\newtheorem*{clproof}{Proof of Claim}


\begin{document}
\tolerance = 9999

\maketitle
\markboth{Tatsuro Kawakami}{Bogomolov-Sommese vanishing and liftability for surface pairs}

\begin{abstract}
We show that the Bogomolov-Sommese vanishing theorem holds for a log canonical projective surface $(X, B)$ in large characteristic unless the Iitaka dimension of $K_X+\lfloor B \rfloor$ is not equal to two.
As an application, we prove that a log resolution of a pair of a normal projective surface and a reduced divisor in large characteristic lifts to the ring of Witt vectors when the Iitaka dimension of the log canonical divisor is less than or equal to zero. 
Moreover, we give explicit and optimal bounds on the characteristic unless their Iitaka dimensions are equal to zero.
\end{abstract}

\tableofcontents

\section{Introduction}
Vanishing theorems involving differential sheaves play a significant role in the analysis of algebraic varieties. 
The Bogomolov-Sommese vanishing theorem, which was originally proved in \cite{Bog}, is one of the most important tools of this kind and has been studied by many authors (see \cite{Gra15}, \cite{GKK}, \cite{GKKP}, \cite{JK}, \cite{SS85} for example).

\begin{thm}[\textup{Bogomolov-Sommese vanishing theorem, \cite[Corollary 1.3]{Gra15}}]\label{BSVoverC}
Let $(X, B)$ be a log canonical (lc, for short) projective pair over the field of complex numbers $\C$.
Then
\[
H^0(X, (\Omega_X^{[i]}(\log\,\lfloor B\rfloor)\otimes \sO_X(-D))^{**})=0
\]
for every $\Z$-divisor $D$ on $X$ satisfying $\kappa(X, D)>i$.
\end{thm}

In Theorem \ref{BSVoverC}, $\kappa(X, D)$ denotes the Iitaka dimension of a $\Z$-divisor $D$, $(-)^{**}$ denotes double dual, and $\Omega_X^{[i]}(\log\,\lfloor B\rfloor)$ denotes the $i$-th logarithmic reflexive differential form of the pair $(X, \lfloor B\rfloor)$, where $\lfloor B\rfloor$ is the round-down of $B$. We refer to Definition \ref{Definition:Iitaka dim} and \textit{Notation} for details.
The logarithmic extension theorem for $(n+1)$-dimensional lc pairs can be deduced from the Bogomolov-Sommese vanishing theorem for $n$-dimensional log Calabi-Yau pairs (see \cite[Section 9]{Gra}). 
The vanishing theorem is also applied to show the vanishing of the second cohomology of the tangent sheaf of lc projective surfaces with big anti-canonical divisors so that they have no local-to-global obstruction (see \cite[Proposition 3.1]{HP}).

In this paper, we discuss an analog of Theorem \ref{BSVoverC} when the pair $(X, B)$ is defined over an algebraically closed field of positive characteristic and $\dim\,X=2$.
It is well-known that the Bogomolov-Sommese vanishing theorem fails when the canonical divisor $K_X$ is big. For example, it is not difficult to see that the sheaf of first differential forms of Raynaud's surface \cite{Ray} contains an ample invertible sheaf. 
Moreover, Langer \cite[Section 8]{Lan} constructed a pair $(S, F)$ of a smooth rational surface $S$ and a disjoint union of smooth rational curves $F$ such that $\Omega_S(\log\,F)$ contains a big invertible sheaf in every characteristic (see also \cite[Section 11]{Langer19}). In other words, the Bogomolov-Sommese vanishing theorem fails even if $X$ is a smooth rational surface.
On the other hand, we can observe that the log canonical divisor $K_S+F$ is big except when the characteristic is equal to two (see Example \ref{Example:Langer's surface} for the details). 
Therefore, it is natural to ask whether the Bogomolov-Sommese vanishing theorem holds when the log canonical divisor is not big and the characteristic is sufficiently large.
We give an affirmative answer to this question.

\begin{thm}\label{BSV, Intro}
There exists a positive integer $p_0$ with the following property. 
Let $(X, B)$ be an lc projective surface pair over an algebraically closed field of characteristic $p>p_0$. 
If $\kappa(X, K_X+\lfloor B \rfloor)\neq 2$, then 
\[
H^0(X, (\Omega_X^{[i]}(\log\,\lfloor B\rfloor)\otimes \sO_X(-D))^{**})=0
\]
for every $\Z$-divisor $D$ on $X$ satisfying $\kappa(X, D)>i$.
Moreover, if $\kappa(X, K_X+\lfloor B \rfloor)=-\infty$ (resp.~$\kappa(X, K_X+\lfloor B \rfloor)=1$), then we can take $p_0=5$ (resp.~$p_0=3$) as an optimal bound. If $\kappa(X, K_X+\lfloor B \rfloor)=0$, then we can take $p_0$ as the maximum number of the Gorenstein index of every klt Calabi-Yau surface over every algebraically closed field.
\end{thm}

In Theorem \ref{BSV, Intro}, a klt Calabi-Yau surface means a klt projective surface whose canonical divisor is numerically trivial.
If the base field is an algebraically closed field of characteristic zero, then the Gorenstein index of a klt Calabi-Yau surface is less than or equal to $21$ by \cite[Theorem C (a)]{Bla}. In general, there exists a uniform bound on the Gorenstein index independent of the choice of the algebraically closed base field (see Lemma \ref{max Gor index}), but its explicit value is not known. 

It is worth pointing out that Theorem \ref{BSV, Intro} is new even when $B=0$ if $X$ is singular. In particular, Theorem \ref{BSV, Intro} shows that lc projective surfaces whose canonical divisors have negative Iitaka dimension have no local-to-global obstruction when the characteristic is bigger than five (see Proposition \ref{tangent} for the details).

The key step of the proof of Theorem \ref{BSV, Intro} is Proposition \ref{ext thm}, which is a generalization of Graf's logarithmic extension theorem \cite[Theorem 1.2]{Gra} for lc surfaces in positive characteristic. 
The Bogomolov-Sommese vanishing theorem in characteristic zero (Theorem \ref{BSVoverC}) is reduced to the case where $(X, B)$ is log smooth by the logarithmic extension theorem (see \cite{Gra} and \cite[7.C. Proof of Theorem 7.2]{GKKP}).
In this reduction process, we need the fact that an index one cover of $D$ is \'etale in codimension one. However, this fact is not necessarily true in characteristic $p>0$ when the Cartier index of $D$ is divisible by $p$. Therefore, we cannot apply Graf's logarithmic extension theorem directly to reduce Theorem \ref{BSV, Intro} to the case where $(X, B)$ is log smooth.
In order to overcome this issue, we prove Proposition \ref{ext thm}. 

Moreover, reducing Theorem \ref{BSV, Intro} to the case of log smooth pairs is not sufficient because the Bogomolov-Sommese vanishing theorem is not known even for such pairs in positive characteristic.
Proposition \ref{ext thm} enables us to apply the logarithmic Akizuki-Nakano vanishing theorem for $W_2$-liftable pairs, which was proved by Hara \cite{Har}, to our log smooth pairs (see Remark \ref{Remark:EXT2} for details).

Next, we fix a log smooth surface pair $(Y, B_Y)$ defined over an algebraically closed field $k$ of characteristic $p>0$.
Let us consider the liftability of $(Y, B_Y)$ to the ring $W(k)$ of Witt vectors.
When $\kappa(Y, K_Y)\leq 0$, $B_Y=0$, and $p>3$, it is well-known that $Y$ lifts to $W(k)$ (see \cite[Proposition 2.6]{Ito}, \cite[Section 11]{Liedtke(Book)}, and \cite[Proposition 11.1]{Oort} for example).
However, when $B_Y\neq 0$, the pair $(Y, B_Y)$ does not have such a lifting in general even when $\kappa(Y, K_Y)=-\infty$. 
Indeed, the counterexample $(S, F)$ to the Bogomolov-Sommese vanishing theorem constructed by Langer does not lift to $W_2(k)$, which means that $(Y, B_Y)$ does not always lift to $W(k)$ even when $Y$ is a smooth rational surface. 
On the other hand, Arvidsson-Bernasconi-Lacini \cite[Theorem 1.2]{ABL} showed that $(Y, B_Y)$ lifts to $W(k)$ when $(Y, B_Y)$ can be realized as some log resolution of a klt projective surface over $k$ of characteristic $p>5$ such that the canonical divisor has negative Iitaka dimension, and in particular, $\kappa(Y, K_Y+B_Y)=-\infty$.
Therefore, it is natural to ask $(Y, B_Y)$ lifts to $W(k)$ when $\kappa(Y, K_Y+B_Y)\leq 0$ and the characteristic $p$ is sufficiently large.
We give a complete answer to this question.

\begin{thm}\label{lift, Intro}
There exists a positive integer $p_0$ with the following property.
Let $X$ be a normal projective surface over an algebraically closed field $k$ of characteristic $p>p_0$, $B$ a reduced divisor on $X$, and
$\pi\colon Y\to X$ a log resolution of $(X, B)$.
Suppose that one of the following holds. 
\begin{enumerate}
    \item[\textup{(1)}] $\kappa(X, K_X+B)=-\infty$.
    \item[\textup{(2)}] $K_X+B\equiv 0$ and $B\neq0$.
    \item[\textup{(3)}] $\kappa(X, K_X+B)=0$.
\end{enumerate}
Then $(Y, \pi_{*}^{-1}B+\Exc(\pi))$ lifts to the ring $W(k)$ of Witt vectors. 
Moreover, when the condition (1) or (2) holds, we can take $p_0=5$ as an optimal bound.
\end{thm}

In Theorem \ref{lift, Intro} (3), $p_0$ should be at least $19$ by Example \ref{Example : sharpness for Kodaira dim 0}, but it is not clear whether we can take $p_0$ as the maximum Gorenstein index of klt Calabi-Yau surfaces. 
In the proof of Theorem \ref{lift, Intro} (1) and (2), we apply Theorem \ref{BSV, Intro} to obtain the vanishing of $H^2(Y, T_Y(-\log\,\pi_{*}^{-1}B+\Exc(\pi)))$, where the obstruction to the lifting lives.
In (3), such a vanishing does not always hold.
Therefore, using an argument of Cascini-Tanaka-Witaszek \cite{CTW}, we show the boundedness of some $\epsilon$-klt log Calabi-Yau surfaces, from which we deduce the desired liftability (see Lemma \ref{Lift of log CY MFS} and Proposition \ref{Lift of strictly klt CY}).

As an application of Theorems \ref{BSV, Intro} and \ref{lift, Intro}, we obtain a Kawamata-Viehweg type vanishing for $\Z$-divisors on a normal projective surface.

\begin{thm}\label{KVV, Intro}
There exists a positive integer $p_0$ with the following property.
Let $X$ be a normal projective surface over an algebraically closed field of characteristic $p>p_0$ and $D$ a nef and big $\Z$-divisor on $X$.
Suppose that one of the following holds.
\begin{enumerate}
    \item[\textup{(1)}] $\kappa(X, K_X)\leq 0$.
    \item[\textup{(2)}] $\kappa(X, K_X)=1$ and $X$ is lc.
\end{enumerate}
Then $H^i(X, \sO_X(K_X+D))=0$ for all $i>0$. 
Moreover, if $\kappa(X, K_X)=-\infty$ (resp.~(2) holds), then we can take $p_0=5$ (resp.~$p_0=3$) as an optimal bound.
\end{thm}

\begin{notation}
A \textit{variety} means an integral separated scheme of finite type over an algebraically closed field. 
A \textit{curve} (resp.~\textit{surface}) means
a variety of dimension one (resp.~two). A \textit{pair} $(X, B)$ consists of a normal variety of $X$ and an effective $\Q$-divisor $B$ with coefficients in $[0,1]\cap\Q$ such that the log canonical divisor $K_X+B$ is $\Q$-Cartier. 
We say a pair $(X,B)$ is \textit{log smooth} when $X$ is smooth and $B$ has a simple normal crossing support.
We refer to \cite[Section 2.3]{KM98} for definitions of singularities appearing in the minimal model program (e.g.~klt, dlt, \ldots).
Throughout this paper, we use the following notation:
\begin{itemize}
\item $\Exc(f)$: the reduced exceptional divisor of a birational morphism $f$.
\item $\lfloor D \rfloor$ (resp.~ $\lceil D \rceil$): the \textit{round-down} (resp.~ \textit{round-up}) of a $\Q$-divisor $D$.
\item $\mathcal{F}^{*}$: the dual of a coherent sheaf of $\mathcal{F}$. 
\item $\Omega_X^{[i]}(\log\,B)$: the \textit{$i$-th logarithmic reflexive differential form} $j_{*}\Omega_U^{i}(\log\,B)$, where $X$ is a normal variety, $B$ is a reduced divisor on $X$, $U$ is the log smooth locus of $(X, B)$, and $j\colon U\hookrightarrow X$ is the natural inclusion morphism.
\item $T_X(-\log\,B)\coloneqq (\Omega_X^{[1]}(\log\,B))^{*}$: the \textit{logarithmic tangent sheaf}, where $X$ is a normal variety and $B$ is a reduced divisor on $X$.
\item $W(k)$ (resp.~$W_n(k)$): the ring of Witt vectors (resp.~the ring of Witt vectors of length $n$), where $k$ is an algebraically closed field of positive characteristic.
\end{itemize}
\end{notation}

\section{Preliminaries}\label{sec:pre}
\subsection{The Iitaka dimension for \texorpdfstring{$\Z$}--divisors}
In this subsection, we recall the definition and basic properties of the Iitaka dimension of $\Z$-divisors. 
\begin{defn}[\textup{\cite[Definition 2.18]{GKKP}}]\label{Definition:Iitaka dim}
Let $X$ be a normal projective variety over an algebraically closed field $k$ and $D$ a $\Z$-divisor on $X$.  
We define the \textit{Iitaka dimension} $\kappa(X,D)\in\{-\infty,0,\cdots,\dim X\}$ of $D$ as follows.

If $H^0(X, \sO_X(mD)) = 0$ for all $m \in \Z_{>0}$, then we say that $D$ has \textit{Iitaka dimension 
$\kappa(X, D)\coloneqq-\infty$}.  Otherwise, set
\[ M \coloneqq \bigl\{ m\in \Z_{>0} \;\big|\; \dim_{k} H^0(X, \sO_X(mD)) > 0 \bigr\}, \]
and consider the natural rational mappings
\[ \phi_m : X \dasharrow \mathbb P\bigl(H^0(X, \sO_X(mD))^*\bigr) \quad \text{ for each } m \in M. \]
Note that we can consider the rational map as above since $\sO_X(mD)$ is invertible on the regular locus of $X$. 
The Iitaka dimension of $D$ is then defined as
\[ \kappa(X, D) \coloneqq \max_{m \in M} \bigl\{ \dim \overline{\phi_m(X)} \bigr\}. \]
When $D$ is a $\Q$-divisor, we define $\kappa(X, D)$ as $\kappa(X, mD)$, where $m$ is any positive integer such that $mD$ is a $\Z$-divisor.
We say a $\Q$-divisor $D$ is \emph{big} if $\kappa(X, D) = \dim X$.
Note that if $D$ is $\Q$-Cartier, then the above definition coincides with the usual definition (\cite[Definition 2.13]{Lazarsfeld}).
\end{defn}

\begin{defn}
Let $D$ be a $\Q$-divisor on a normal projective surface $X$.
We say $D$ is \textit{nef} if $D\cdot C\geq0$ for every curve $C$ on $X$.
We refer to \cite[Chapter 14, 14.24]{Bad} for the definition of the intersection number of $\Z$-divisors on a normal projective surface.
\end{defn}

\begin{rem}\label{nefness and bigness}\,
\begin{itemize}
    \item Let $f\colon Y\to X$ be a projective birational morphism of normal surfaces. We recall the \textit{Mumford pullback} by $f$. If $Y$ is smooth, then we refer to \cite[Chapter 14, 14.24]{Bad}. When $Y$ is not smooth, then we take a resolution
    $\pi\colon \tilde{Y}\to Y$. Then $\Supp(\pi^{*}\Exc(f))\subset \Exc(f\circ \pi)$, and thus the intersection matrix of $\Exc(f)$ is negative definite. Now, we can define the Mumford pullback by $f$ in a similar way to \cite[Chapter 14, 14.24]{Bad}.
    \item The Mumford pullback preserves the Iitaka dimension by the projection formula. In addition, the Mumford pullback preserves nefness by definition.
\end{itemize}
\end{rem}

In the rest of the paper, we simply refer to the Mumford pullback as pullback.

\subsection{Liftability of pairs to the ring of Witt vectors}
In this subsection, we recall the fundamental facts about liftability of pairs.
\begin{defn}\label{def:snc}
Let $T$ be a Noetherian scheme.
Let $X$ be a smooth scheme over $T$ of relative dimension $d$ and $B_1,\ldots, B_n$ irreducible closed subschemes. We say that $B\coloneqq \sum_{r=1}^{n} B_r$ is \textit{simple normal crossing over $T$} (\textit{snc over $T$}, for short) if the scheme-theoretic intersection $\bigcap_{r \in J} B_r$ is smooth over $T$ of relative dimension $d-|J|$ for any subset $J \subseteq \{1, \ldots, n\}$ such that $\bigcap_{r \in J} B_r\neq \emptyset$. When $T$ is a spectrum of an algebraically closed field, we simply say that $B$ is \textit{snc}. 
\end{defn}
\begin{rem}
We follow the notation of Definition \ref{def:snc} and suppose that $B$ is snc over $T$.
By \cite[Theorem 2.5.8]{MR2791606}, for any $x \in \bigcap_{r \in J} B_r$, there exist an open neighbourhood $U_x\subset X$ of $x$ and an \'etale morphism 
$\phi_x \colon U_x \to \mathbb A^d_T = T \times_{\Spec \Z} \Spec \Z[t_1, ..., t_d]$ such that $(\bigcap_{r \in J} B_r)|_{U_x} = V(\prod_{r\in J}\phi_{x}^{\#}(t_r))$. 
\end{rem}

\begin{defn}\label{d-liftable}
Let $X$ be a smooth projective variety over an algebraically closed field $k$ and $B$ an snc divisor on $X$. 
Let $B = \sum_{r=1}^n B_r$ be the irreducible decomposition.
Let $R$ be a Noetherian local ring with residue field $k$.
We say that a pair $(X,B)$ \textit{lifts to $R$} if there exist 
\begin{itemize}
    \item a scheme $\mathcal{X}$ smooth and projective over $R$ with a closed immersion $i\colon X\hookrightarrow \mathcal{X}$ and 
    \item irreducible closed subschemes $\mathcal{B}_1,\ldots, \mathcal{B}_n$ such that $\sum_{r=1}^{n}\mathcal{B}_{r}$ is snc over $R$
\end{itemize}
such that the induced morphism  $i\times_{R}k \colon X\to \mathcal{X}\times_{R} k$ is isomorphic and $(i\times_{R}k) (B_r)= \mathcal{B}_r\times_{R} k$ for every $1 \leq r\leq n$. 
\end{defn}

\begin{rem}\label{Remark:lifting as log smooth pairs}
In the setting of Definition \ref{d-liftable}, we further assume that $R$ is regular. In this case, if $\sum_{r=1}^{n}\mathcal{B}_{r}$ is flat over $R$, then $\sum_{r=1}^{n}\mathcal{B}_{r}$ is snc over $R$ as follows.

Since $\mathcal{B}_r$ is flat over $\Spec\,R$, this is smooth of relative dimension $d-1$ by \cite[Th\'eor\`eme 12.2.4 (iii)]{EGAIV3}. 
We fix a closed point of $x\in \bigcap_{r \in J} \mathcal{B}_r$.
Since $\mathcal{X}$ is regular, each $\mathcal{B}_r$ is Cartier, and thus we obtain 
\begin{align*}
&\dim\sO_{\mathcal{X},x}-|J|\leq \dim\,\sO_{\bigcap_{r \in J} \mathcal{B}_r, x}
\leq\dim\,\sO_{\bigcap_{r \in J} B_r, x}+\dim\,R\\
=&\dim\,\sO_{X,x}-|J|+\dim\,R
=\dim\sO_{\mathcal{X},x}-|J|
\end{align*}
and hence $\sO_{\bigcap_{r \in J} \mathcal{B}_r, x}$ is Cohen-Macaulay and \[\dim\,\sO_{\bigcap_{r \in J} \mathcal{B}_r, x}=\dim\,\sO_{\bigcap_{r \in J} B_r, x}+\dim\,R\]
holds.
Then, by \cite[Theorem 23.1]{Matsumura}, it follows that $\bigcap_{r \in J} \mathcal{B}_r\to \Spec\,R$ is flat at $x$. By \cite[Th\'eor\`eme 12.2.4 (iii)]{EGAIV3} again, we conclude that the closed subscheme $\bigcap_{r \in J} \mathcal{B}_r$ is smooth over $R$ and hence $\mathcal{B}$ is snc over $R$. 
\end{rem}

\begin{lem}\label{log lift of blow-up}
Let $X$ be a smooth projective surface and $B$ an snc divisor on $X$.
Let $\mathrm{Bl}_{x}\colon Y\to X$ be a blow-up at a closed point $x\in X$. Suppose that $(X, B)$ lifts to a complete regular local ring $R$.
Then $(Y, (\mathrm{Bl}_{x})_{*}^{-1}B+\Exc(\mathrm{Bl}_{x}))$ lifts to $R$.
\end{lem}
\begin{proof}
Let $(\mathcal{X}, \mathcal{B})$ be a lifting of $(X,B)$ to $R$.
Since $R$ is henselian, \cite[Propostion 2.8.13]{MR2791606} shows that there exists a lifting $\tilde{x}$ of $x$ to $R$, which is compatible with the snc structure in the sense of \cite[Definition 2.7]{ABL}.
By \cite[Theorem 2.5.8]{MR2791606},
there exists an open neighborhood $\mathcal{U}_{\tilde{x}}\subset \mathcal{X}$ of $\tilde{x}$ and an \'etale morphism $\tilde{\phi}\colon\mathcal{U}_{\tilde{x}}\to \Spec\,R[t_1,t_2]$ such that 
\begin{itemize}
    \item $\mathcal{B}_r\cap \mathcal{U}_{\tilde{x}}=V(\tilde{\phi}^{*}t_r)$ for each irreducible component $\mathcal{B}_r$ of $\mathcal{B}$ containing $\tilde{x}$ and
    \item $\tilde{x}=V(\tilde{\phi}^{*}t_1,\tilde{\phi}^{*}t_2)$.
\end{itemize}
We define an \'etale morphism $\phi\colon U_x\coloneqq \mathcal{U}_{\tilde{x}}\otimes_R k \to \Spec\,k[t_1,t_2]$ as $\phi\coloneqq\tilde{\phi}\otimes_R k$.
Then $x=V(\phi^{*}t_1, \phi^{*}t_2)$.
Now, an argument of after Claim of \cite[Lemma 4.4]{Ard} shows that $(\mathcal{Y}, (\mathrm{Bl}_{\tilde{x}})_{*}^{-1}\mathcal{B}+\Exc(\mathrm{Bl}_{\tilde{x}}))$ is a lifting of $(Y, (\mathrm{Bl}_{x})_{*}^{-1}B+\Exc(\mathrm{Bl}_{x}))$.
\end{proof}

\begin{lem}\label{every lift=some lift}
Let $X$ be a normal projective surface and $B$ a reduced divisor on $X$.
Suppose that there exists a log resolution $f\colon Y\to X$ of $(X, B)$ such that $H^2(Y, \sO_Y)=0$ and $(Y, f_{*}^{-1}B+\Exc(f))$ lifts to a complete regular local ring $R$.
Then, for every log resolution $g\colon Z\to X$ of $(X, B)$, the pair $(Z, g_{*}^{-1}B+\Exc(g))$ lifts to $R$.
\end{lem}
\begin{proof}
We take a log resolution $g\colon Z\to X$ of $(X, B)$ and show the liftability of $(Z, g_{*}^{-1}B+\Exc(g))$.
We can take a log resolution $h\colon W\to X$ of $(X, B)$ that factors through both $f$ and $g$.
Since $W\to Y$ is a composition of blow-ups at a smooth point, the pair $(W, h_{*}^{-1}B+\Exc(h))$ lifts to $R$ by Lemma \ref{log lift of blow-up}.
Since $W\to Z$ is also a composition of blow-ups at a smooth point, it follows from \cite[Proposition 4.3 (1)]{AZ} that $(Z, g_{*}^{-1}B+\Exc(g))$ formally lifts to $R$. By assumption, we have $H^2(Z, \sO_Z)=H^2(Y, \sO_Y)=0$, and hence $(Z, g_{*}^{-1}B+\Exc(g))$ lifts to $R$ as a scheme.
\end{proof}

\begin{thm}\label{log smooth lift criterion}
Let $X$ be a smooth projective surface over an algebraically closed field $k$ and $B$ an snc divisor on $X$.
Let $R$ be a Noetherian complete local ring with residue field $k$.
Suppose that $H^2(X, T_{X}(-\log\,B))=0$.
Then $(X ,B)$ lifts to $R$ as a formal scheme.
Moreover, if $H^2(X, \sO_X)=0$, then $(X,B)$ lifts to $R$ as a scheme.
\end{thm}
\begin{proof}
We refer to \cite[Theorem 2.3]{KN} for the proof.
\end{proof}

Hara \cite[Corollary 3.8]{Hara98} showed the logarithmic Akizuki-Nakano vanishing theorem for $W_2$-liftable pairs $(X, B)$.
In Theorem \ref{Hara's vanishing}, we slightly generalize this theorem to the vanishing for nef and big divisors when $\dim\,X=2$.

\begin{thm}[\textup{cf.~\cite[Corollary 3.8]{Hara98}}]\label{Hara's vanishing}
Let $X$ be a smooth projective surface over an algebraically closed field $k$ of characteristic $p>0$ and $B$ an snc divisor on $X$. 
Suppose that $(X, B)$ lifts to $W_2(k)$.
Let $D$ be a nef and big $\Q$-divisor on $X$ such that $\Supp(\lceil D\rceil-D)$ is contained in $B$. 
Then 
\[
H^{j}(X, \Omega_X^{i}(\log \, B)\otimes \sO_X(-\lceil D \rceil))=0
\]
for $i,j\in \Z_{\geq 0}$ such that $i+j<2$.
\end{thm}

\begin{rem}
Langer \cite[Example 1]{Lan15} showed that Theorem \ref{Hara's vanishing} does not hold when $D$ is only big. In other words, the Bogomolov-Sommese vanishing theorem can fail on $W_2(k)$-liftable pairs.
\end{rem}

\begin{proof}
By the Serre duality and the essentially same argument as in \cite[Corollary 3.8]{Hara98}, we can reduce the assertion to 
\[
H^j(X, \Omega_X^i(\log \, B))\otimes \sO_X(-B+\lceil p^eD \rceil)=0
\]
for all $i+j>2$ and some $e>0$.
We remark that the assumption that $p>\dim\,X$ in \cite[Corollary 3.8]{Hara98} is relaxed to $p\geq \dim\,X$. Indeed, in the proof of \cite[Corollary 3.8]{Hara98}, the assumption that $p>\dim\,X$ is only used for the quasi-isomorphism
$\bigoplus_{i}\Omega_X^i(\log\, B)[-i]\cong F_{*}\Omega_X^{\bullet}(\log B)$, and this quasi-isomorphism holds even in $p=\dim\,X$ by \cite[10.19 Proposition]{EV}.

We take $m, n\in \Z_{>0}$ such that $p^m(p^n-1)D$ is Cartier.
Then we obtain 
\begin{align*}
&H^j(X, \Omega_X^i(\log \, B))\otimes \sO_Y(-B+\lceil p^{m+ln}D \rceil)\\
=&H^j(X, \Omega_X^i(\log \, B)\otimes\sO_Y(-B+\lceil p^mD \rceil+(\sum_{i=0}^{l-1}p^{ni})p^m(p^n-1)D)).
\end{align*}
When $j=2$, 
the last term vanishes for all sufficiently large $l\gg 0$ by \cite[Proposition 2.3]{Tan15}.
Moreover, when $(i, j)=(2, 1)$, the last term vanishes for all sufficiently large $l\gg 0$ by \cite[Theorem 2.6]{Tan15}.
Therefore, we obtain the desired vanishing.
\end{proof}

\begin{lem}\label{Lem:KVV}
Let $X$ be a normal projective surface over an algebraically closed field $k$ of positive characteristic and $D$ a nef and big $\Z$-divisor on $X$. Suppose that there exists a log resolution $\pi\colon Y\to X$ such that $(Y, \Exc(\pi))$ lifts to $W_2(k)$.
Then $H^i(X, \sO_X(K_X+D))=0$ for all $i>0$.
\end{lem}
\begin{proof}
By the Serre duality for Cohen-Macaulay sheaves (\cite[Theorem 5.71]{KM98}), it suffices to show that $H^i(X, \sO_X(-D))=0$ for all $i<2$. When $i=0$, the vanishing follows from the bigness of $D$. Thus we assume that $i=1$.
By the spectral sequence 
\[
E_2^{p,q}=H^{p}(X,R^{q}\pi_{*}\sO_{Y}(-\lceil \pi^{*}D\rceil ))\Rightarrow E^{p+q}=H^{p+q}(Y,\sO_{Y}(-\lceil \pi^{*}D \rceil)),
\]
we obtain an injective morphism \[
H^{1}(X,\pi_{*}\sO_{Y}(-\lceil \pi^{*}D\rceil))\hookrightarrow H^1(Y, \sO_{Y}(-\lceil \pi^{*}D\rceil)).\]
By the projection formula, we have 
$\pi_{*}\sO_{Y}(-\lceil \pi^{*}D\rceil )=\pi_{*}\sO_{Y}(\lfloor -\pi^{*}D\rfloor )=\sO_{X}(-D)$
and thus it suffices to show that $H^1(Y, \sO_{Y}(-\lceil\pi^{*}D\rceil))=0$. 
Since $\Supp(\lceil\pi^{*}D\rceil-\pi^{*}D)\subset\Exc(\pi)$ and $\pi^{*}D$ is nef and big (see Remark \ref{nefness and bigness}), we obtain the desired vanishing by Theorem \ref{Hara's vanishing}.
\end{proof}

\section{Klt Calabi-Yau surfaces}\label{sec:klt CY surfaces}
In this section, we prove the liftability of a log resolution of a klt Calabi-Yau surface in large characteristic (Propositions \ref{Prop : Lift of canonical CY} and \ref{Lift of strictly klt CY}). 
We also show that there exists a bound on the Gorenstein index for every klt Calabi-Yau surface over every algebraically closed field (Lemma \ref{max Gor index}). 

\begin{defn}
We fix a real number $\epsilon\in\R_{>0}$.
We say a pair $(X, B)$ is \textit{$\epsilon$-klt} if, for every proper birational morphism $f\colon Z\to X$ from a normal variety $Z$, all the coefficients of $f^{*}(K_X+B)-K_Z$ is less than $1-\epsilon$.
\end{defn}

\begin{defn}
We say that a normal projective variety $X$ is \textit{canonical (resp.~klt) Calabi-Yau} if $X$ has only canonical (resp.~klt) singularities and $K_X\equiv0$. Moreover, if $X$ is klt Calabi-Yau but not canonical Calabi-Yau, then we say that $X$ is \textit{strictly klt Calabi-Yau}.
    
We say a projective pair $(X,B)$ is \textit{log Calabi-Yau}
if $(X, B)$ is lc and $K_X+B\equiv 0$.
We say a normal projective variety $X$ is \textit{of Calabi-Yau type} if there exists an effective $\Q$-divisor $B$ such that $(X, B)$ is log Calabi-Yau. 
\end{defn}

First, we show the liftability of a log resolution of a canonical Calabi-Yau surface.

\begin{prop}\label{Prop : Lift of canonical CY}
Let $X$ be a canonical Calabi-Yau surface over an algebraically closed field $k$ of characteristic $p>19$.
Then, for every log resolution $f\colon Z\to X$ of $X$, the pair $(Z, \Exc(f))$ lifts to $W(k)$. 
\end{prop}
\begin{proof}
Let $\pi\colon Y\to X$ be the minimal resolution.
By Lemma \ref{log lift of blow-up}, it suffices to show the liftability of $(Y, E\coloneqq\Exc(\pi))$.
Since $K_Y=\pi^{*}K_X=0$, it follows that $Y$ is one of an abelian surface, a hyperelliptic surface, a K3 surface, or an Enriques surface.
If $Y$ is an abelian surface, then $Y=X$ and $Y$ lifts to $W(k)$ by \cite[Proposition 11.1]{Oort}.
Next, we assume that $Y$ is a hyperelliptic surface. In this case, $Y=X$ and $Y$ is a quotient of a fiber product $C_1\times C_2$ of elliptic curves by an action of some group scheme $G$.  
We recall that a smooth projective curve lifts to $W(k)$ with its automorphism if the degree of the automorphism is not divisible by $p$ (\cite[Theorem 1.5 and Remark 1.11]{Obus}).
Since $p\neq 2,3$, comparing with the list of actions of $G$ on $C_1\times C_2$ in \cite[List 10.27]{Bad}, we can take a $W(k)$-lifting $\mathcal{C}_i$ of $C_i$ and $\mathcal{G}$ of $G$ such that $\mathcal{G}$ acts on $\mathcal{C}_1\times \mathcal{C}_2$ compatibly with the action of $G$ on $C_1\times C_2$. Then $\mathcal{C}_1\times \mathcal{C}_2/\mathcal{G}$ gives a lifting of $Y$.

Next, we assume that $Y$ is a K3 surface or an Enriques surface.
We show that the determinant $d$ of the intersection matrix of $E$ is not divisible by $p$.
For the sake of contradiction, we assume that $d$ is divisible by $p$. 
Since the determinant of the intersection matrix of a rational double point of type $A_n$ (resp.~$D_n$, $E_6$, $E_7$, $E_8$) is equal to $(-1)^n(n+1)$ (resp.~$(-1)^n4,~3,-2,1$), it follows from the assumption of $p>19$ that $X$ has an $A_{np-1}$-singularity for some $n\in\Z_{>0}$.
Hence we have $\rho(Y)\geq np\geq 23$.
This is a contradiction because the Picard rank of a K3 surface is at most $22$ by \cite[Chapter 17, 2.4]{K3book} and that of an Enriques surface is at most $10$ by \cite[Section 3]{BMIII}.
Thus $d$ is not divisible by $p$ and \cite[Theorems 1.2 and 1.3]{Gra} shows that $\pi_{*}\Omega_Y=\Omega^{[1]}_{X}$.
Then we obtain
\begin{align*}
H^2(Y, T_Y(-\log\,E))\hookrightarrow H^2(X, T_X)\cong& H^0(X, \Omega_X^{[1]}\otimes \sO_X(K_X))\\=&H^0(Y, \Omega_Y\otimes \sO_Y(K_Y)).
\end{align*}
For the first injection, we refer to Remark \ref{Remark:tangent}.
We assume that $Y$ is a K3 surface. Then we have $H^0(Y, \Omega_Y\otimes \sO_Y(K_Y))=H^0(Y, \Omega_Y)=0$, and $(Y, E)$ formally lifts to $W(k)$ by Theorem \ref{log smooth lift criterion}. Moreover, $(Y, E)$ is algebraizable by \cite[Proposition 2.6]{Ito}.
Finally, We assume that $Y$ is an Enriques surface. Then we have an \'etale morphism $\tau \colon \tilde{Y}\to Y$ from a K3 surface $\tilde{Y}$ since $p\neq 2$.
Thus we obtain \[
H^0(Y, \Omega_Y\otimes \sO_Y(K_Y))\hookrightarrow H^0(\tilde{Y}, \Omega_{\tilde{Y}}\otimes \sO_{\tilde{Y}}(K_{\tilde{Y}}))=0.
\]
Moreover, since $p\neq 2$, we have $K_Y\neq 0$, and in particular, $H^2(Y, \sO_Y)\cong H^0(Y, \sO_Y(K_Y))= 0$. Therefore, the pair $(Y, E)$ lifts to $W(k)$ by Theorem \ref{log smooth lift criterion}.
\end{proof}

\begin{rem}\label{Rem : Lift of canonical CY}
In Proposition \ref{Prop : Lift of canonical CY}, we cannot drop the assumption $p>19$ (see Example \ref{Example : sharpness for Kodaira dim 0}). On the other hand, when the minimal resolution $Y$ is a K3 surface that is not supersingular, the pair $(Y,E)$ lifts to $W(k)$ even if $p\leq 19$ as follows. 

First, by \cite[Corollary 4.2]{LM18}, $Y$ itself lifts to $W(k)$.
Moreover, by \cite[Lemma 2.3 and Corollary 4.2]{LM18}, each irreducible component of $E$ lifts to $W(k)$. Then we obtain the desired liftability by Remark \ref{Remark:lifting as log smooth pairs}. 
\end{rem}

From now, we focus on a strictly klt Calabi-Yau surface.
We first prove that the Gorenstein index of a klt Calabi-Yau surface is bounded from above.

\begin{lem}\label{Cartier index divides det}
Let $X$ be a klt surface and $\pi\colon Y\to X$ a resolution.
Then the Cartier index of any $\Z$-divisor on $X$ divides the determinant of the intersection matrix of $\Exc(\pi)$.
\end{lem}
\begin{proof}
Let $d$ be the determinant of the intersection matrix of $\Exc(\pi)$.
We take a $\Z$-divisor $D$ on $X$ and write $\pi^{*}D=\pi_{*}^{-1}D+\sum d_iE_i$ for some $d_i\in \Q$. 
Then it follows that $dd_i\in \Z$ for each $i$, and in particular, $\pi^{*}dD$ is Cartier.
Now, we can conclude that $dD$ is Cartier by \cite[Lemma 2.1]{CTW}.
\end{proof}

\begin{lem}\label{Q-fac index of epsilon CY}
We fix a real number $\epsilon\in (0, \frac{1}{\sqrt{3}})$.
Then there exists $m\coloneqq m(\epsilon)\in \Z_{>0}$ with the following property.
For every $\epsilon$-klt surface $X$ of Calabi-Yau type over every algebraically closed field and every $\Z$-divisor $D$ on $X$, the divisor $mD$ is Cartier.
\end{lem}
\begin{proof}
Let $\pi\colon Y\to X$ be the minimal resolution and $\Exc(\pi)\coloneqq \sum_i E_i$ the irreducible decomposition.
Then $Y$ is $\epsilon$-klt and of Calabi-Yau type.
By \cite[Lemma 1.2 and Theorem 1.8]{AM}, we have $-\frac{2}{\epsilon}\leq E_i^2\leq -2$ and $\rho(Y)\leq \frac{128}{\epsilon^5}$. In addition, we have $E_i\cdot E_j=0$ or $1$ for $i\neq j$ since $X$ is klt. 
Thus there are only finitely many possibilities for the intersection matrix of $\Exc(\pi)$. 
We take $m$ as a product of all possible determinants of the intersection matrices of $\Exc(\pi)$.
Now, Lemma \ref{Cartier index divides det} shows that $m$ is the desired integer.
\end{proof}

\begin{lem}\label{Boundedness of Vol of epsilon CY}
We fix a real number $\epsilon\in (0, \frac{1}{\sqrt{3}})$.
For every $\epsilon$-klt surface $X$ of Calabi-Yau type over every algebraically closed field, there are only finitely many possibilities for $K_X^2$.
\end{lem}
\begin{proof}
Let $\pi\colon Y\to X$ be the minimal resolution.
We can write $K_Y+\sum_i a_iE_i=\pi^{*}K_X$
for some $a_i\in \Q_{>0}$, where $E_i$ is a $\pi$-exceptional prime divisor. 
As in the proof of Lemma \ref{Q-fac index of epsilon CY}, we have $\rho(Y)\leq \frac{128}{\epsilon^5}$ and there are only finitely many possibilities for the intersection matrix of $\Exc(\pi)$. 
We fix a positive integer $m\coloneqq m(\epsilon)\in\Z_{>0}$ as in Lemma \ref{Q-fac index of epsilon CY}. 
Then we have $a_i\in\{\frac{1}{m},\cdots,\frac{m-1}{m}\}$ for each $i$.

If $Y$ is rational, then $Y$ is obtain from $\PP^2_k$ or a Hirzebruch surface by at most $(\lfloor\frac{128}{\epsilon^5}\rfloor-1)$-times blow-ups, and in particular, $K_Y^2\in \Z\cap (9-\lfloor\frac{128}{\epsilon^5}\rfloor,9)$.
If $Y$ is not rational, then $K_Y^2=0$ by \cite[Lemma 1.4]{AM}.
Now, we can conclude that there are only finitely many possibilities for 
\[
K_X^2=K_Y^2+\sum_i a_i(K_Y\cdot E_i)=K_Y^2+\sum_i a_i(-E_i^2-2) 
\] 
and obtain the assertion.
\end{proof}

\begin{lem}[\textup{cf.~\cite[Proposition 11.7]{Bir}}]\label{global ACC}
Let $\Lambda\subset [0,1]\cap \Q$ be a DCC set. Then there exists a finite subset $\Gamma\subset \Lambda$ with the following property:
for every projective morphism $X\to Z$ over every algebraically closed field and every $\Q$-divisor $B$ on $X$ satisfying 
\begin{itemize}
    \item $(X,B)$ is an lc surface,
    \item the coefficients of $B$ are in $\Lambda$,
    \item $K_X+B$ is numerically trivial over $Z$, and
    \item $\dim\,X>\dim\,Z$,
\end{itemize}
all the $\pi$-horizontal coefficients of $B$ are contained in $\Gamma$.
\end{lem}
\begin{proof}
The assertion has been proved in \cite[Proposition 11.7]{Bir} when we fix the base field. We remark that the same proof works without fixing the base field.
We note that, in Step~4 of the proof of \cite[Proposition 11.7]{Bir}, we use \cite[Theorem 6.9]{Ale}, which requires us to fix the base field.
However, \cite[Theorem 6.9]{Ale} is applied to only show the boundedness of the Gorenstein index and the self-intersection number of the canonical divisor of an $\epsilon$-klt del Pezzo surface, which do not depend on the base field by Lemmas \ref{Q-fac index of epsilon CY} and \ref{Boundedness of Vol of epsilon CY}. 
\end{proof}

\begin{lem}\label{epsilon-klt=klt}
There exists a positive real number $\epsilon \in \R_{>0}$ such that every klt Calabi-Yau surface over every algebraically closed field is $\epsilon$-klt.
\end{lem}
\begin{proof}
First, we extract an exceptional divisor with minimum log discrepancy.
We take a klt Calabi-Yau surface $X$ as in the lemma.
Let $\pi\colon Y\to X$ be the minimal resolution and write \[
K_{Y}+\sum_{i} a_{X,i}E_{i}=\pi^{*}K_X
\]
for some $a_{X, i}\in \Q_{>0}$. 
We may assume that $a_{X,1}\geq a_{X,i}$ for all $i$.
We run a $(K_Y+a_{X,1}E_{1}+\sum_{i\geq 2} E_{i})$-MMP over $X$ to obtain a birational contraction $\phi\colon Y\to Y'$.
Since $K_Y+a_{X,1}E_{1}+\sum_{i\geq 2} E_{i}\equiv_{X}\sum_{i\geq 2}(1-a_{X,i})E_{i}$, it follows that $\phi_{*}E_1\neq 0$ and $\sum_{i\geq 2}(1-a_{X,i})\phi_{*}E_{i}$ is nef over $X$. The negativity lemma shows that $\phi_{*}E_{i}=0$ for each $i\geq 2$ and hence
\[
K_{Y'}+a_{X,1}\phi_{*}E_{1}\equiv \phi_{*}(K_{Y}+\sum_{i} a_{X,i}E_{i})\equiv 0.
\]

Now, we prove the assertion.
For the sake of contradiction, we assume that there exists a sequence of klt Calabi-Yau surfaces $\{X_m\}_{m\in\Z_{>0}}$ such that $\{{a_{X_m,1}}\}_{m\in\Z_{>0}}$ is a strictly increase sequence.
Since $\{{a_{X_m,1}}|m\in\Z_{>0}\}$ is a DCC set, we can derive a contradiction by Lemma \ref{global ACC}. 
\end{proof}

\begin{lem}\label{max Gor index}
There exists a minimum positive integer $n\in\Z_{>0}$ such that, for every klt Calabi-Yau surface $X$ over every algebraically closed field, the Gorenstein index of $X$ is less than or equal to $n$.
\end{lem}
\begin{proof}
The assertion follows from Lemmas \ref{Q-fac index of epsilon CY} and \ref{epsilon-klt=klt}.
\end{proof}

\begin{rem}\label{Rem:max Gor index}
There exists a klt Calabi-Yau surface over $\C$ whose Gorenstein index is $19$ by \cite[Theorem C (a)]{Bla}. Thus we have $n\geq 19$ in Lemma \ref{max Gor index}.
Moreover, \cite[Theorem C (a)]{Bla} also shows that we can take $n=21$ when the base field of $X$ is an algebraically closed fields of characteristic zero.
\end{rem}

\begin{lem}\label{Cartier index}
Let $X$ be a strictly klt Calabi-Yau surface and $n$ the Gorenstein index of $X$.
Then $n$ is a minimum positive integer such that $nK_X=0$.
\end{lem}
\begin{proof}
By the abundance theorem (\cite[Theorem 1.2]{Tan12}), we can take a minimum positive integer $l$ such that $lK_X=0$.
By the definition, we have $n \leq l$. We show that $l \leq n$. Let $\pi\colon Y \to X$ be the minimal resolution of $X$.
Then we have $nK_Y+E= \pi^{*}nK_X\equiv 0$ for some effective Cartier divisor $E$. Since $X$ is strictly klt Calabi-Yau, it follows from the proof of \cite[Lemma 1.4]{AM} that $Y$ is a rational surface. Thus numerically trivial Cartier divisors on $Y$ are linearly trivial, and in particular, $nK_Y+E=0$.
Now we obtain $nK_X=\pi_{*}(nK_Y+E)=0$ and hence $l\leq n$.  
\end{proof}

Lemmas \ref{max Gor index} and \ref{Cartier index} show that a global cyclic cover associated to the canonical divisor of a strictly klt Calabi-Yau surface is \'etale in codimension one in large characteristic.

Finally, we prove the liftability of a log resolution of a strictly klt Calabi-Yau surface in large characteristic. 

\begin{defn}\label{MFS}
Let $(X, B)$ be a pair and $f \colon X \to Z$ a projective surjective morphism to a normal variety $Z$.
We say $f \colon X \to Z$ is a $(K_X+B)$-\emph{Mori fiber space} if
\begin{itemize}
    \item $-(K_X+B)$ is $f$-ample,
    \item $f_{*}\sO_X=\sO_Z$ and $\dim \, X>\dim \, Z$, and
    \item the relative Picard rank $\rho(X/Z)=1$.
\end{itemize}
\end{defn}

\begin{lem}\label{Lift of log CY MFS}
We fix a finite set $I\subset [0,1)\cap \Q$ and a positive real number $\epsilon\in (0, \frac{1}{\sqrt{3}})$.
There exists a positive integer $p(I, \epsilon)\in\Z_{>0}$ with the following property.
Let $(X, B)$ be an $\epsilon$-klt log Calabi-Yau surface over an algebraically closed field $k$ of characteristic bigger than $p(I, \epsilon)$. Suppose that $X$ admits a $K_X$-Mori fiber space structure $f\colon X\to Z$ and all the coefficients of $B$ are contained in $I$. Then, for every log resolution $g\colon W\to X$ of $(X,B)$, the pair $(W, g_{*}^{-1}(\Supp(B))+\Exc(g))$ lifts to $W(k)$.
\end{lem}
\begin{proof}
By \cite[Proposition 2.5]{ABL} and Lemma \ref{every lift=some lift}, it suffices to show the following liftability:
there exists a log resolution $g\colon W\to X$ of $(X, B)$ such that the pair $(W, g_{*}^{-1}(\Supp(B))+\Exc(g))$ lifts to characteristic zero over a smooth base in the sense of \cite[Definition 2.15]{CTW}.
When $\dim\,Z=0$, by replacing $B$ with $\frac{1}{2}B$, we can assume that $(X, B)$ is an $\epsilon$-klt log del Pezzo surface, and the assertion follows from \cite[Proposition 3.2]{CTW}. Therefore, we may assume that $\dim\,Z=1$.

We show the following claim.
\begin{cl}
There exists a flat family $(\mathcal{X}, \mathcal{B})\to T$ to a reduced quasi-projective scheme $T$ over $\Spec\,\Z$ such that every log Calabi-Yau surface $(X, B)$ over every algebraically closed field of characteristic bigger than five satisfying
\begin{itemize}
    \item $(X,B)$ is $\epsilon$-klt, 
    \item $X$ has a $K_X$-Mori fiber structure $f\colon X\to Z$ to a curve $Z$, and
    \item all the coefficients of $B$ are contained in $I$,
\end{itemize} 
is a geometric fiber of $(\mathcal{X}, \mathcal{B})\to T$.
\end{cl}
\begin{clproof}
As in the proof of \cite[Lemma 3.1]{CTW}, it suffices to show the following:
there exists a positive integer $m\in\Z_{>0}$ not depending on $X$ and a very ample divisor $H_X$ on $X$ such that 
\begin{itemize}
    \item $mB$ is Cartier and 
    \item there are only finitely many possibilities for $\dim_{k} H^0(X, \sO_X(H_X))$, $H_X^2$, $H_X\cdot K_X$, $H_X\cdot B$, $K_X\cdot B$, and $B^2$ 
\end{itemize} 
for every log Calabi-Yau surface $(X,B)$ as in the claim.
We take a positive integer $m=m(\epsilon)$ as in Lemma \ref{Q-fac index of epsilon CY}.
Since all the coefficients of $B$ are contained in a finite set $I$, we can assume that $mB$ is Cartier, and thus the first assertion holds.

We show the latter assertion.
Together with $B\equiv -K_X$ and Lemma \ref{Boundedness of Vol of epsilon CY}, it suffices to check the values of $\dim_{k} H^0(X, \sO_X(H_X)), H_X^2$, and  $H_X\cdot K_X$.
We first show that $A_X\coloneqq -K_X+(\lceil \frac{2}{\epsilon}\rceil-1)F$ is an ample Cartier divisor on $X$ such that there are only finitely many possibilities for $A_X^2, A_X\cdot K_X$, where $F$ is a fiber of $X\to Z$.
Let $C$ be an irreducible curve whose numerical class spans an extremal ray of $\overline{NE}(X)$ that is not spanned by the numerical class of $F$.
If $C^2\geq 0$, then we have 
\[
(-K_{X}+(\lceil \frac{2}{\epsilon}\rceil-1)F)\cdot C=(B+(\lceil \frac{2}{\epsilon}\rceil-1)F)\cdot C
\geq \lceil \frac{2}{\epsilon}\rceil-1>0,
\]
and thus $-K_X+(\lceil \frac{2}{\epsilon}\rceil-1)F$ is ample by Kleiman's ampleness criterion.
We next assume that $C^{2}<0$.
Let $\pi\colon Y\to X$ be the minimal resolution. Then $Y$ is $\epsilon$-klt and of Calabi-Yau type, and thus \cite[Lemma 1.2]{AM} shows that $-\frac{2}{\epsilon}\leq (\pi_{*}^{-1}C)^2$. In particular, $-\frac{2}{\epsilon}\leq C^2$.
Now, we have 
\begin{align*}
(-K_{X}+(\lceil \frac{2}{\epsilon}\rceil-1)F)\cdot C
=(B+(\lceil\frac{2}{\epsilon}\rceil-1)F)\cdot C
>& (1-\epsilon)C^2+\lceil \frac{2}{\epsilon}\rceil-1\\
\geq&-\frac{2}{\epsilon}+2+\lceil\frac{2}{\epsilon}\rceil-1>0,
\end{align*}
and hence $-K_X+(\lceil\frac{2}{\epsilon}\rceil-1)F$ is ample.
Together with $F^2=0$, $K_{X}\cdot F=-2$, and Lemma \ref{Boundedness of Vol of epsilon CY}, we can see that $A_X$ is the desired ample Cartier divisor.  

Now, by \cite[Theorem 1.2]{Wit}, it follows that $13mK_{X}+45mA_X$ is very ample. 
Moreover, we can see that $(13m-3)K_X+(45m-14)A_X$ is nef and hence $H^i(X, \sO_{X}(13mK_{X}+45mA_X))=0$ for all $i>0$ by \cite[Proposition 6.5]{Wit}. 
We set $H_X\coloneqq 13mK_{X}+45mA_X$.
Then there are only finitely many possibilities for $H_X^2$ and $H_X\cdot K_X$.
Moreover, by the Riemann-Roch theorem, we have 
\begin{align*}
&\dim_{k} H^0(X, \sO_X(H_X))=\mathcal{X}(\sO_{X}(H_X))
                 =\mathcal{X}(\sO_{W}(f^{*}H_X)\\
                 =&\frac{(f^{*}H_X)^2}{2}+\frac{f^{*}H_X\cdot (-K_{W})}{2}+1
                 =\frac{(H_X)^2}{2}+\frac{H_X\cdot (-K_{X})}{2}+1,
\end{align*}
where $f\colon W\to X$ is a resolution and we used the fact that $X$ has only rational singularities for the second equality.
Therefore, there are only finitely many possibilities for $\dim_{k} H^0(X, \sO_X(H_X))$, and we finish the proof of the claim.
\end{clproof}

Now, by the above claim and the proof of \cite[Proposition 3.2]{CTW}, we can find the desired positive integer $p(I,\epsilon)$.
\end{proof}

\begin{prop}\label{Lift of strictly klt CY}
There exists a positive integer $p_0$ with the following property. 
Let $X$ be a strictly klt Calabi-Yau surface over an algebraically closed field of characteristic $p>p_0$.
Then, for every log resolution $g \colon W\to X$, the pair $(W, \Exc(g))$ lifts to $W(k)$.
\end{prop}
\begin{proof}
By Lemma \ref{epsilon-klt=klt}, there exists a positive real number $\epsilon\in (0, \frac{1}{\sqrt{3}})$ such that every klt Calabi-Yau surface is $\epsilon$-klt. 
We take $m=m(\epsilon)$ as in Lemma \ref{Q-fac index of epsilon CY} and define a finite set $I\coloneqq \{\frac{1}{m},\cdots,\frac{m-1}{m}\}$.
We take $p_0\coloneqq p(I,\epsilon)$ as in Lemma \ref{Lift of log CY MFS}.

Let $X$ be a strictly klt Calabi-Yau surface over an algebraically closed field of characteristic $p>p_0$.
As in Lemma \ref{epsilon-klt=klt}, we can take an extraction $f \colon Y\to X$ of an exceptional prime divisor $E_1$ such that $a_1\coloneqq\mathrm{coeff}_{E_1}(f^{*}K_X-K_{Y})\in I$.
Since $K_Y\equiv -a_1E_1$ is not pseudo-effective, we can run a $K_{Y}$-MMP to obtain a birational contraction $\phi\colon Y\to Y'$ and a $K_{Y'}$-Mori fiber space $Y'$.
Since $K_Y+a_1E_1\equiv 0$, the negativity lemma shows that $K_Y+a_1E_1=\phi^{*}(K_{Y'}+a_1E'_1)$ and hence $(Y', a_1E'_1)$ is $\epsilon$-klt and log Calabi-Yau, where $E'_1\coloneqq \phi_{*}E_1$. 
Then, by Lemma \ref{Lift of log CY MFS} and the definition of $p_0$, we can take a log resolution $\mu\colon Z\to Y'$ of $(Y', a_1E'_1)$ that factors through $\phi$ and $(Z, \mu_{*}^{-1}E'_1+\Exc(\mu))$ lifts to $W(k)$.
We now have the following diagram:
\[
\xymatrix{
 (Z, \mu_{*}^{-1}E'_1+\Exc(\mu))\ar[rd]_\mu\ar[r]_-h & (Y, E_1) \ar[r]_{f}\ar[d]^{\phi}& X \\
  &   (Y', E'_1)     &. \\
}
\]
Since $\Exc(f\circ h)\subset \mu_{*}^{-1}E'_1+\Exc(\mu)$, the pair $(Z, \Exc(f\circ h))$ lifts to $W(k)$, and the assertion holds by Lemma \ref{every lift=some lift}.
\end{proof}

\section{The Bogomolov-Sommese vanishing theorem}\label{sec:BSV}
\subsection{An extension type theorem for lc surfaces}
In this subsection, we show an extension type theorem for lc surfaces (Proposition \ref{ext thm}), which plays an essential role in the proof of Theorem \ref{BSV, Intro}.

\begin{lem}[\textup{cf.~\cite[Lemma 2.2]{Kaw1}}]\label{push}
Let $f\colon Y\to X$ be a projective birational morphism of normal surfaces. 
Let $B_Y$ be a reduced $\Z$-divisor and $D_Y$ a $\Z$-divisor on $Y$.
We set $B\coloneqq f_{*}B_Y$ and $D\coloneqq f_{*}D_Y$.
Then the following hold.
\begin{enumerate}
    \item[\textup{(1)}] The natural restriction morphism \[f_{*}(\Omega_{Y}^{[1]}(\log\,B_Y)\otimes \sO_{Y}(-D_Y))^{**}\to (\Omega_X^{[1]}(\log\, B)\otimes \sO_X(-D))^{**}
    \]
    is injective.
    \item[\textup{(2)}] Suppose that $X$ and $Y$ are projective. Then $\kappa(Y, D_Y)\leq \kappa(X, D)$ holds. 
    \end{enumerate}
\end{lem}
\begin{proof}
The same proof as \cite[Lemma 2.2]{Kaw1} works. Note that (1) is local on $X$.
\end{proof}
\begin{rem}\label{Remark:tangent}
In the setting of Lemma \ref{push} (2), 
we have
\begin{align*}
H^2(Y, T_Y(-\log\, B_Y))\cong&\mathrm{Hom}_{\sO_Y}(T_Y(-\log\, B_Y), \sO_Y(K_Y))\\
                      \cong &H^0(Y, (\Omega_Y^{[1]}(\log\,B_Y)\otimes \sO_Y(K_Y))^{**})
\end{align*}
by the Serre duality.
Then, by Lemma \ref{push} (1), we obtain an injective morphism
\[
H^2(Y, T_Y(-\log\, B_Y))\hookrightarrow H^2(X, T_X(-\log\, B)).
\]
We will use this fact in Section \ref{sec : Lift of a surface pair}.
\end{rem}

\begin{defn}\label{Definition:dlt blow-up}
Let $X$ be a normal surface and $B$ a $\Q$-divisor with coefficients in $[0,1]$.
We say that a morphism $h\colon W\to X$ is a \textit{dlt blow-up of $(X, B)$} if 
\begin{enumerate}
    \item[\textup{(1)}] $h$ is a projective birational morphism,
    \item[\textup{(2)}] $(W, h^{-1}_{*}B+\Exc(h))$ is dlt, and
    \item[\textup{(3)}] $K_W+h^{-1}_{*}B+\Exc(h)+F=h^{*}(K_X+B)$ for some effective $\Q$-divisor $F$. 
    \end{enumerate}
\end{defn}

\begin{lem}\label{Lemma:dlt blow-up}
Let $X$ be a normal surface and $B$ is a $\Q$-divisor with coefficients in $[0,1]$. Then the following hold.
\begin{enumerate}
    \item[\textup{(1)}] Any log resolution $\pi\colon Y\to X$ of $(X, B)$ decomposes into a birational projective morphism $Y\to W$ and a dlt blow-up $W\to X$. 
    \item[\textup{(2)}] $F=0$ if and only if $(X, B)$ is lc.
\end{enumerate}
\end{lem}
\begin{proof}
We refer to \cite[Theorem 4.7 and Remark 4.8]{Tanaka(exc)} for the proof.
Note that, using the Mumford pullback, \cite[Remark 4.8 (1)]{Tanaka(exc)} holds without the assumption that $K_X+B$ is $\R$-Cartier. 
\end{proof}

\begin{defn}
Let $(X, B)$ be a dlt pair over an algebraically closed field of characteristic $p>0$ such that $B$ is reduced.
We say that $(X, B)$ is \textit{tamely dlt} if the Cartier index of $K_X+B$ is not divisible by $p$.
\end{defn}

\begin{defn}\label{def:tame decomp}
Let $(X,B)$ be a pair over an algebraically closed field of positive characteristic such that $B$ is reduced.
Let $\pi\colon Y\to X$ be a log resolution of $(X,B)$. 
We say $\pi\colon Y\to X$ admits a \textit{tame decomposition} if there exists a decomposition 
\[
\pi\colon Y_0\coloneqq Y\overset{f_0}{\to} Y_1 \to \cdots \overset{f_{m-1}}{\to} Y_m\coloneqq X
\]
of $\pi$ such that 
\begin{itemize}
    \item $(Y_i, B_{Y_i})$ is tamely dlt and
    \item $-(K_{Y_i}+B_{Y_i})$ is $f_i$-nef, 
\end{itemize}
for all $i\in\{0,\dots,m-1\}$, where $B_{Y_0}=B_Y\coloneqq f_{*}^{-1}B+\Exc(\pi)$ and $B_{Y_{i}}$ is the pushforward of $B_Y$ to $Y_i$.
We remark that a tame resolution \cite[Definition 7.1]{Gra} admits a tame decomposition.
\end{defn}

\begin{lem}[cf.~\textup{\cite[7.B. Proof of Theorem 1.2]{Gra}}]\label{Lemma:tame decomposition}
Let $(X, B)$ be an lc surface over an algebraically closed field of characteristic $p>5$ and $\pi\colon Y\to X$ a log resolution of $(X,B)$.
Then $\pi$ admits a tame decomposition.
\end{lem}
\begin{proof}
We follow the notation of Definition \ref{def:tame decomp}.
We note that the minimal resolution of lc surface singularities with reduced boundary are classified as in (7.8.1)--(7.8.7) in \cite[7.B. Proof of Theorem 1.2]{Gra}.
By the dual graph of them (\cite[Figure 2--9]{Gra}), we can see that every $(-1)$-curve $F$ on $Y$ intersects at most two components of $\Supp(B_Y)$.
Then \[(K_Y+B_Y)\cdot F=(K_Y+F)\cdot F + (B_Y-F)\cdot F\leq 0,\] and by taking $f_0$ as the contraction of all $(-1)$-curves on $Y$, we can assume that $\pi\colon Y\to X$ is the minimal resolution.

For the cases of (7.8.1)--(7.8.4), (7.8.6) with non-zero boundary, and (7.8.7) in \cite[7.B. Proof of Theorem 1.2]{Gra}, it has been already proved that $\pi$ is a tame resolution, and in particular, admits a tame decomposition. We use here $p>5$ essentially.
For the case of (7.8.5), $-(K_Y+B_Y)$ is $\pi$-nef since the dual graph of the $\pi$-exceptional divisor is a chain, and thus $\pi$ itself is a tame decomposition.
Finally, we discuss the case (7.8.6) with zero boundary. Let $f_0 \colon Y\to Y_1$ be the contraction of two $(-2)$-curves $G_1$ and $G_2$ which intersect the fork.
Then $(K_Y+B_Y)\cdot G_i=-1$ and $Y_1$ is tamely dlt since $p\neq 2$.
Next, let $f_1\colon Y_1\to X$ be the contraction of all the remaining $\pi$-exceptional divisors, whose dual graph is a chain.
Then we can see that $f_{1}\circ f_{0}$ is a tame decomposition of $\pi$.
 \end{proof}

\begin{prop}[An extension type theorem for lc surfaces]\label{ext thm}
Let $(X, B)$ be an lc surface pair over an algebraically closed field of characteristic $p>5$ and $D$ a $\Z$-divisor on $X$. 
Let $f\colon Y\to X$ be a projective birational morphism from a normal surface $Y$ and $B_Y\coloneqq f_{*}^{-1}B+\Exc(f)$.
Then the natural restriction morphism 
\[
\Phi\colon
 f_{*}(\Omega_{Y}^{[1]}(\log\,\lfloor B_Y\rfloor)\otimes \sO_{Y}(-\lceil f^{*}D \rceil))^{**}\to (\Omega_X^{[1]}(\log\,\lfloor B\rfloor)\otimes \sO_X(-D))^{**}
\]
is isomorphic. 
\end{prop}
\begin{rem}\label{Remark:EXT}
Proposition \ref{ext thm} is equivalent to saying that 
\[
f_{*}(\Omega_{Y}^{[1]}(\log\,\lfloor B_Y\rfloor)\otimes \sO_{Y}(-\lceil f^{*}D \rceil))^{**}
\]
is reflexive.
\end{rem}
\begin{rem}\label{Remark:EXT2}
If we take $D=0$ in the proposition, then this is nothing but Graf's logarithmic extension theorem (\cite[Theorem 1.2]{Gra}).
Let us see why it is necessary to generalize Graf's logarithmic extension theorem as Proposition \ref{ext thm} to prove Theorem \ref{BSV, Intro}.

First, we consider the case when the base field is an algebraically closed field of characteristic zero and follow the notation of Theorem \ref{BSVoverC}. Let $\pi\colon Y\to X$ be a log resolution and $B_Y\coloneqq \pi^{-1}_{*}B+\Exc(\pi)$.
Suppose that there exists a $\Z$-divisor $D$ and an injective morphism $\sO_X(D)\hookrightarrow \Omega^{[i]}_X(\log\,B)$. For simplicity, we assume that $D$ is $\Q$-Cartier.
Then, by applying the logarithmic extension theorem in characteristic zero \cite[Theorem 1.5]{GKKP}, we can construct a $\Z$-divisor $D_Y$ on $Y$ such that there exists an injective morphism $\sO_Y(D_Y)\hookrightarrow \Omega^{[i]}_Y(\log\,B_Y)$ and $\kappa(X, D)=\kappa(Y, D_Y)$. This means the Bogomolov-Sommese vanishing theorem can be reduced to the case of log smooth pairs by the logarithmic extension theorem (see \cite[7.C. Proof of Theorem 7.2.]{GKKP} for the detailed argument).
In the construction of $D_Y$, we use the fact that an index one cover of $D$ is \'etale in codimension. However, when we work in characteristic $p>0$ and the Cartier index of $D$ is divisible by $p$, this fact is not always true. Therefore, we cannot apply Graf's logarithmic extension theorem directly to reduce Theorem \ref{BSV, Intro} to the case where $(X, B)$ is log smooth. 

Moreover, in positive characteristic, reducing to the case of log smooth surfaces is not enough because the Bogomolov-Sommese vanishing theorem is not known even for such pairs. Proposition \ref{ext thm} asserts that $D_Y$ can be taken as $\lceil f^{*}D \rceil$, and this enables us to apply the Akizuki-Nakano vanishing theorem (Theorem \ref{Hara's vanishing}) when $D$ is ample. 
\end{rem}
\begin{proof}[Proof of Proposition \ref{ext thm}]
\textbf{Step~0.}
Throughout the proof of this proposition, $\Omega_W^{[1]}(\log\,B_W)(-D_W)$ denotes
 $(\Omega_W^{[1]}(\log\,B_W)\otimes \sO_W(-D_W))^{**}$ for every surface pair $(W, B_W)$ and $\Z$-divisor $D_W$.
By Lemma \ref{push} (1), $\Phi$ is injective.
Since $(X, \lfloor B \rfloor)$ is lc (see \cite[Proposition 7.2]{Gra}), by replacing $B$ with $\lfloor B \rfloor$, we may assume that $B$ is reduced.
Moreover, since the assertion of the proposition is local on $X$, we may assume that $X$ is affine.
Therefore, it suffices to show that
\[
\Phi\colon H^0(Y, \Omega_{Y}^{[1]}(\log\,B_Y)(-\lceil f^{*}D \rceil))\hookrightarrow H^0(X, \Omega_X^{[1]}(\log\,B)(-D))
\]
is surjective.

\textbf{Step~1.}
First, we prove the following claim.
\begin{cl}\label{step}
Suppose that 
\begin{itemize}
    \item $(Y, B_Y)$ is tamely dlt and
    \item $-(K_Y+B_Y)$ is $f$-nef.
\end{itemize}
Then $\Phi$ is surjective.
\end{cl}
\begin{clproof}
We take $s\in H^0(X, \Omega_X^{[1]}(\log\, B)(-D))$. We construct a section of $H^0(Y, \Omega_{Y}^{[1]}(\log\, B_Y)(-\lceil \pi^{*}D \rceil))$ that maps to $s$ by $\Phi$.
We may assume that $s$ is non-zero, and thus $s$ is considered as an injective $\sO_X$-module homomorphism $s\colon \sO_X(D) \hookrightarrow \Omega_X^{[1]}(\log\,B)$.
By \cite[Theorem 6.1]{Gra}, the natural restriction morphism $f_{*}\Omega_Y^{[1]}(\log\,B_Y)\cong \Omega_X^{[1]}(\log\,B)$ is isomorphic.
Then we have a generically injective $\sO_Y$-module homomorphism
\[
f^{*}\sO_X(D) \overset{f^{*}s}{\to} f^{*}\Omega_X^{[1]}(\log\,B)\cong f^{*}f_{*}\Omega_{Y}^{[1]}(\log\,B_Y) \to \Omega^{[1]}_{Y}(\log\,B_Y).
\]
By taking double dual, we obtain an injective $\sO_Y$-module homomorphism
\[
s_Y\colon f^{[*]}\sO_X(D) \hookrightarrow \Omega^{[1]}_{Y}(\log\,B_Y),
\]
where $f^{[*]}\sO_X(D)\coloneqq (f^{*}\sO_X(D))^{**}$.
We take a $\Z$-divisor $D_Y$ on $Y$ such that $\sO_{Y}(D_Y)= f^{[*]}\sO_X(D)$.
Since $\sO_X(f_*D_Y)=(f_{*}\sO_Y(D_Y))^{**}=(f_{*}f^{[*]}\sO_X(D))^{**}=\sO_X(D)$, it follows that $f_{*}D_Y$ is linearly equivalent to $D$. 
By replacing $D_Y$ with $D_Y+f^{*}(D-f_{*}D_Y)$, we may assume that $f_*D_Y=D$.
In particular, $D_Y-f^{*}D$ is $f$-exceptional.

Now, we replace $D_Y$ so that $D_Y-f^{*}D$ is effective.
We assume that $D_Y-f^{*}D$ is not effective. 
By applying the negativity lemma to the negative coefficients part of $D_Y-f^{*}D$, we can take a prime $f$-exceptional divisor $E_1$ such that $\mathrm{mult}_{E_1}(D_Y-f^{*}D)<0$ and $D_Y\cdot E_1>0$.
Then we can show that $s_Y$ factors though an injective $\sO_Y$-module homomorphism $\sO_{Y}(D_Y+E_1)\hookrightarrow \Omega^{[1]}_{Y}(\log\,B_Y)$.
This follows from the essentially same argument as \cite[Theorem 6.1]{Gra}, but we provide the proof here for the completeness.  

Since $(Y, B_Y)$ is tamely dlt, we have the following commutative diagram \begin{equation*}
\xymatrix{ & & \sO_{Y}(D_Y) \ar@{.>}[ld] \ar[d]^-{s_Y} \ar[rd]^-{t} &\\
                 0\ar[r] &\Omega^{[1]}_{Y}(\log\,B_Y-E_1)\ar[r]   & \Omega^{[1]}_{Y}(\log\,B_Y) \ar[r]^{\res_{E_1}}  & \sO_{E_1} \ar[r] & 0,}
\end{equation*}
and a surjective morphism 
\[
\res^m_{E_1}\colon \Sym^{[m]}\Omega_Y^{[1]}(\log\,B_Y)\coloneqq (\Sym^{m}\Omega_Y^{[1]}(\log\,B_Y))^{**}\to \sO_{E_1}
\]
for each $m>0$ that coincides with $\Sym^{m}(\res_{E_1})$ at the generic point of $E_1$ by \cite[Theorem 1.4 (1.4.1)]{Gra}. 
We show that $t$ is zero. For the sake of contradiction, we assume that $t$ is not zero. Since $\Im(t)\subset \sO_{E_1}$ is a torsion-free $\sO_{E_1}$-module, it follows that $t$ is non-zero at the generic point of $E_1$ and so is $\Sym^{m}(t)\colon\sO_Y(D_Y)^{\otimes m}\to \sO_{E_1}$. Since $\Sym^{m}(s_Y)$ and $\Sym^{m}(\res_{E_1})$ coincides with $\Sym^{[m]}(s_Y)\coloneqq (\Sym^m(s_Y))^{**}$ and $\res^m_{E_1}$ at the generic point of $E_1$ respectively, the composition $\res^m_{E_1}\circ\Sym^{[m]}(s_Y)$ coincides with $\Sym^m(t)=\Sym^m(\res_{E_1})\circ\Sym^m(s_Y)$, and in particular, is non-zero at the generic point of $E_1$. Now, we fix $m>0$ such that $mD_Y$ is Cartier. Note that $Y$ is $\Q$-factorial since $(Y, B_Y)$ is dlt.
By restricting $\res^m_{E_1}\circ\Sym^{[m]}(s_Y)$ to $E_1$, we obtain an injective $\sO_{E_1}$-module homomorphism $\sO_{E_1}(mD_Y)\hookrightarrow \sO_{E_1}$ and hence $0<mD_Y\cdot E_1=\deg(\sO_{E_1}(mD_Y))\leq 0$, a contradiction. Therefore $t$ is zero and the morphism $s_Y$ factors through $\sO_{Y}(D_Y) \to \Omega^{[1]}_{Y}(\log\,B_Y-E_1)$.
Then, by \cite[Theorem 1.5 (1.5.1)]{Gra}, we obtain the following commutative diagram
\begin{equation*}
\xymatrix{ & & \sO_{Y}(D_Y) \ar@{.>}[ld] \ar[d]^{s_Y} \ar[rd]^{v} &\\
                 0\ar[r] &\Omega^{[1]}_{Y}(\log\,B_Y)(-E_1)\ar[r]   & \Omega^{[1]}_{Y}(\log\,B_Y-E_1) \ar[r]^-{\mathrm{restr_{E_1}}}  & \omega_{E_1}(\lfloor E_1^c \rfloor) \ar[r] & 0,}
\end{equation*}
and a surjective morphism 
\[
\mathrm{restr}^{m}_{E_1}\colon \Sym^{[m]}\Omega^{[1]}_{Y}(\log\,(B_Y-E_1))\to \sO_{E_1}(mK_{E_1}+\lfloor mE_1^c \rfloor)
\]
that coincides with $\Sym^{m}(\mathrm{restr_{E_1}})$ at the generic point $E_1$.
Here, $E_1^c$ denotes the different $\mathrm{Diff}_{E_1}(B_Y-E_1)$ (see \cite[Definition 4.2]{Kol13} for the definition).
Since $-(K_Y+B_Y)$ is $f$-nef, it follows that 
\[
\deg(\sO_{E_1}(mK_{E_1}+\lfloor mE_1^c \rfloor))\leq (mK_Y+mB_Y)\cdot E_1\leq 0
\]
for all $m>0$ and hence an argument similar to above shows that $v=0$ and $s_Y$ factors through $\sO_{Y}(D_Y) \hookrightarrow \Omega^{[1]}_{Y}(\log\,B_Y)(-E_1)$. 
In particular, we obtain an injective $\sO_Y$-module homomorphism $\sO_{Y}(D_Y+E_1)\hookrightarrow \Omega^{[1]}_{Y}(\log\,B_Y)$ that coincides with $s_Y$ on $Y\setminus\Exc(f)$.
By replacing $D_Y$ with $D_Y+E_1$, and repeating the above procedure, we can assume that $D_Y-f^{*}D$ is effective.

Now, we obtain a $\Z$-divisor $D_Y$ on $Y$ such that $D_Y-f^{*}D\geq 0$ and a morphism $s_Y \in H^0(Y, \Omega_{Y}^{[1]}(\log\, B_Y)(-D_Y))$, which maps to $s$ under the natural restriction morphism 
\[
\Phi' \colon H^0(Y, \Omega_{Y}^{[1]}(\log\, B_Y)(-D_Y))\to H^0(X, \Omega_X^{[1]}(\log\,B)(-D)).
\]

Since $\lceil f^{*}D \rceil\leq \lceil D_Y \rceil=D_Y$,
it follows that $\Phi'$ decomposes into the natural injective morphism
\[
\Theta \colon H^0(Y, \Omega_{Y}^{[1]}(\log\, B_Y)(-D_Y))\hookrightarrow H^0(Y, (\Omega_{Y}^{[1]}(\log\, B_Y)(-\lceil f^{*}D \rceil))
\]
and the morphism 
\[
\Phi\colon H^0(Y, \Omega_{Y}^{[1]}(\log\,B_Y)(-\lceil f^{*}D \rceil))\hookrightarrow H^0(X, \Omega_X^{[1]}(\log\,B)(-D)).
\]
Now we have $\Phi(\Theta(s_Y))=\Phi'(s_Y)=s$ and hence $\Phi$ is surjective. Thus we finish the proof of the claim.
\end{clproof}

\textbf{Step~2.}
Next, we confirm that we may assume that $f\colon Y\to X$ is a log resolution of $(X, B)$.
Let $\tilde{f}\colon Z\to X$ be a log resolution of $(X,B)$ that factors through $f$.
Suppose that the natural restriction morphism
\[
\Phi_{Z,X} \colon H^0(Z, \Omega_{Z}(\log\,B_Z)(-\lceil \tilde{f}^{*}D \rceil))\hookrightarrow H^0(X, \Omega_X^{[1]}(\log\,B)(-D))
\]
is surjective, where $B_Z\coloneqq \tilde{f}^{-1}_{*}B+\Exc(\tilde{f})$. 
Since $\Phi_{Z,X}$ factors through the natural restriction
\[
\Phi_{Y,X} \colon H^0(Y, \Omega^{[1]}_{Y}(\log\,B_Y)(-\lceil f^{*}D \rceil))\hookrightarrow H^0(X, \Omega_X^{[1]}(\log\,B)(-D)),
\]
the restriction $\Phi_{Y,X}$ is also surjective.
Thus we may assume that $f\colon Y\to X$ is a log resolution of $(X, B)$.

\textbf{Step~3.}
Finally, we show that the surjectivity of $\Phi$ and finish the proof of the proposition. By Lemma \ref{Lemma:tame decomposition}, a log resolution $f$ admits a tame decomposition
\[
f\colon Y_0\coloneqq Y\overset{f_1}{\to} Y_1 \to \cdots \overset{f_m}{\to} Y_m\coloneqq X.
\]

By the claim in Step~1, the natural restriction morphisms
\begin{align*}
\Phi_{m-1,m} \colon &H^0(Y_{m-1}, \Omega_{Y_{m-1}}^{[1]}(\log\,B_{Y_{m-1}})(-\lceil f_{m}^{*}D \rceil))\cong H^0(X, \Omega_X^{[1]}(\log\,B)(-D)),\\
\Phi_{m-2,m-1} \colon &H^0(Y_{m-2}, \Omega_{Y_{m-2}}^{[1]}(\log\,B_{Y_{m-2}})(-\lceil f_{m-1}^{*}\lceil f_{m}^{*}D \rceil \rceil))\\\cong &H^0(Y_{m-1}, \Omega_{Y_{m-1}}^{[1]}(\log\,B_{Y_{m-1}})(-\lceil f_{m}^{*}D \rceil))
\end{align*}
are isomorphic.
Then $\Phi_{m-1,m}\circ\Phi_{m-2,m-1}$ factors through the natural restriction morphism
\begin{align*}
\Phi_{m-2,m}\colon &H^0(Y_{m-2}, \Omega_{Y_{m-2}}^{[1]}(\log\,B_{Y_{m-2}})(-\lceil f_{m-1}^{*}f_{m}^{*}D\rceil))\\
\hookrightarrow &H^0(X, \Omega_X^{[1]}(\log\,B)(-D)),
\end{align*}
and hence $\Phi_{m-2,m}$ is isomorphic.
By repeating this procedure, we can conclude that $\Phi$ is an isomorphism. 
\end{proof}

\subsection{Proof of Theorem \ref{BSV, Intro}}
In this subsection, we prove Theorem \ref{BSV, Intro}.
First, we show the Bogomolov-Sommese vanishing theorem on a surface admitting a fibration structure including a Mori fiber space and an lc trivial fibration.

\begin{lem}\label{LCTF}
Let $X$ be a normal surface over an algebraically closed field $k$ of characteristic $p>3$ and $B$ a reduced divisor on $X$.  Let $f\colon X\to Z$ be a projective surjective morphism such that $\dim\,Z=1$, $f_{*}\sO_X=\sO_Z$, and $-(K_X+B)$ is $f$-nef.
Then 
\[
f_{*}(\Omega_X^{[1]}(\log\,B)\otimes \sO_X(-D))^{**}=0
\]
for every $\Z$-divisor $D$ satisfying $D\cdot F>0$ for a general fiber $F$ of $f$.
\end{lem}
\begin{proof}
Since $f_{*}(\Omega_X^{[1]}(\log\,B)\otimes \sO_X(-D))^{**}$ is torsion-free, it suffices to show that the rank of the sheaf is zero, and in particular, we can shrink $Z$ for the proof.
First, we prove the following claim.

\begin{cl}
By shrinking $Z$, we may assume that $B$ is snc over $Z$.
\end{cl}
\begin{clproof}
We note that a general fiber $F$ is reduced and irreducible since $\dim\,Z=1$ and $f_{*}\sO_X=\sO_Z$.
By shrinking $Z$, we may assume that all irreducible components of $B$ dominant $Z$.
Let $n\in\Z_{\geq0}$ be the number of the irreducible components of $B$.
Then we have 
\[
\deg(K_F)=K_X\cdot F\leq K_X\cdot F+n\leq (K_X+B)\cdot F\leq0
\]
and hence $(\deg(K_F), n)=(0, 0), (-2, 0), (-2, 1)$, or $(-2, 2)$.
If $(\deg(F), n)=(0, 0)$, then $B=0$ and $F$ is an elliptic curve since $p>3$. 
Similarly, if $(\deg(F), n)=(-2, 0)$, then $B=0$ and $F\cong \PP_k^1$. 
Next, if $(\deg(F), n)=(-2, 1)$, then $F\cong \PP_k^1$ and $B\cdot F=1$ or $2$. In the case where $B\cdot F=1$, it follows that $B$ and $F$ intersect transversally. 
In the case where $B\cdot F=2$, the restricted morphism $f|_{B}\colon B\to Z$ is generically \'etale since $p\neq 2$.  
Finally, if $(\deg(F), n)=(-2, 2)$, then $B_1\cdot F=B_2\cdot F=1$ and thus both $B_1$ and $B_2$ intersect transversally with $F$, where $B_1$ and $B_2$ are irreducible components of $B$.

Therefore, in each case, we can assume that $B$ is snc over $Z$ by shrinking $Z$ and finish the proof of the claim.
\end{clproof}

Now, we show that the assertion of the lemma.
We shrink $Z$ so that $Z$ is affine and $B$ is snc over $Z$.
Note that $(X, B)$ is log smooth in this case.
For the sake of contradiction, we assume that 
\[
H^0(X, \Omega_X(\log\,B)\otimes \sO_X(-D))\neq 0
\]
for some $\Z$-divisor $D$ satisfying $D\cdot F>0$.
Then there exists an injective $\sO_X$-module homomorphism $s\colon \sO_X(D)\hookrightarrow \Omega_X(\log\,B)$.
Since $B$ is snc over $Z$, we have the following diagram:
\begin{equation*}
\xymatrix{ & & \sO_{X}(D) \ar@{.>}[ld] \ar[d]^{s} \ar[rd]^{t} &\\
                 0\ar[r] &\sO_X(f^{*}K_Z)\ar[r]   & \Omega_X(\log\,B) \ar[r]  & \Omega_{X/Z}(\log\, B) \ar[r] & 0.}
\end{equation*}
 In the above diagram, when $B\neq 0$, we define $\Omega_X(\log\,B) \to \Omega_{X/Z}(\log\, B)$ by $d(f^{*}{z})\mapsto 0, dx/x\mapsto dx/x$, where $z$ is a coordinate on $Z$ and $x$ is a local equation of $B$. Note that $f^{*}{z}$ and $x$ form coordinates on $X$ since $B$ is snc over $Z$. When $B=0$, this is the usual relative differential sequence for $f$ (\cite[II Proposition 8.11]{Har}).
Suppose that $t$ is non-zero. Then, by restricting $t$ to $F$, we have an injective $\sO_F$-module homomorphism $t|_{F}\colon \sO_F(D) \hookrightarrow \Omega_{F}(\log\, B|_F)=\sO_F(K_F+B_F)$, where the injectivity follows from the generality of $F$. This shows that 
\[
0<\deg (D|_F)\leq\deg(K_F+B|_F)=(K_X+B)\cdot F\leq0,
\]
a contradiction.
Thus $t$ is zero and the morphism $s$ factors through $\sO_X(D)\hookrightarrow \sO_X(f^{*}K_Z)$.
Then by considering the restriction to $F$, we obtain
\[
0<\deg (D|_F)\leq\deg(f^*K_Z|_F)=0,
\]
a contradiction.
Hence we conclude that $H^0(X, \Omega_X(\log\,B)\otimes \sO_X(-D))=0$.
\end{proof}

Now, we prove Theorem \ref{BSV, Intro}.

\begin{proof}[Proof of Theorem \ref{BSV, Intro}]
\textbf{Step~0.}
By replacing $B$ with $\lfloor B\rfloor$, we may assume that $B$ is reduced. 
Since the assertion is obvious when $i=0$ or $2$, it suffices to show that 
\begin{equation}
    H^0(X, (\Omega_{X}^{[1]}(\log\,B)\otimes \sO_{X}(-D))^{**})=0 \tag{a}
\end{equation}
for every big $\Z$-divisor $D$.
Let $h\colon (W, B_W\coloneqq h^{-1}_*B+\Exc(h))\to (X,B)$ be a dlt blow-up. 
Then $\kappa(W, K_W+B_W)=\kappa(X, K_X+B)$ and 
the vanishing (a) is equivalent to saying that
\begin{equation}
H^0(W, (\Omega_{W}^{[1]}(\log\,B_W)\otimes \sO_{W}(-\lceil h^{*}D\rceil))^{**})=0 \tag{b}
\end{equation}
when $p>5$ by Proposition \ref{ext thm}.
We set $D_W\coloneqq \lceil h^{*}D\rceil$. By Remark \ref{nefness and bigness}, $D_W$ is big.

\textbf{Step~1.} 
First, we assume that $\kappa(X, K_X+B)=-\infty$ and $p>5$. We show the vanishing (b).
In this case, $K_W+B_W$ is not pseudo-effective by the abundance theorem (\cite[Theorem 1.2]{Tan12}).
By Lemma \ref{push} (1) and (2), we can replace $W$ with an output of a $(K_W+B_W)$-MMP and assume that $W$ has a $(K_W+B_W)$-Mori fiber space structure $f\colon W\to Z$. 
If $\dim\,Z=1$, then the assertion follows from Lemma \ref{LCTF}.
Thus we assume that $\dim\,Z=0$. In this case, $W$ is a klt del Pezzo surface of Picard rank one and $D_W$ is an ample $\Q$-Cartier $\Z$-divisor.
Let $\pi\colon Y\to W$ be a log resolution of $(W, B_W)$, $B'\coloneqq \pi^{-1}_{*}B_W$, $E\coloneqq \Exc(\pi)$, and $B_Y\coloneqq B'+E$.
Then by Proposition \ref{ext thm}, it suffices to show that 
\[
H^0(Y, \Omega_{Y}(\log\,B_Y)\otimes \sO_{Y}(-\lceil \pi^{*}D_W \rceil))=0.
\]
For the sake of contradiction, we assume that there exists an injective $\sO_Y$-module homomorphism $s\colon \sO_Y(\lceil \pi^{*}D_W \rceil))\hookrightarrow \Omega_{Y}(\log\,B_Y)$.
We show that $s$ factors through $s\colon \sO_Y(\lceil \pi^{*}D_W \rceil))\hookrightarrow \Omega_{Y}(\log\,E)$.
Let $B'_1$ be an irreducible component of $B'$.
Since $(Y, B_Y)$ is log smooth, we obtain the following diagram
\begin{equation*}
\xymatrix{ & & \sO_{Y}(\lceil \pi^{*}D_W\rceil) \ar@{.>}[ld] \ar[d]^{s} \ar[rd]^{t} &\\
                 0\ar[r] &\Omega_{Y}(\log\,B_Y-B'_1)\ar[r]   & \Omega_{Y}(\log\,B_Y) \ar[r]  & \sO_{B'_1} \ar[r] & 0.}
\end{equation*}
Since $B'_1$ is not $\pi$-exceptional and $D_W$ is an ample $\Q$-Cartier $\Z$-divisor, it follows that 
\[
\lceil \pi^{*}D_W\rceil \cdot B'_1\geq \pi^{*}D_W\cdot B'_1= D_W\cdot \pi_{*}B'_1>0,
\] 
and hence $t$ is zero.
Thus the morphism $s$ factors through $\sO_Y(\lceil \pi^{*}D_W\rceil)\hookrightarrow \Omega_{Y}(\log\,B_Y-B'_1)$. By repeating this procedure, we can show that $s$ factors through $\sO_Y(\lceil \pi^{*}D_W\rceil)\hookrightarrow \Omega_{Y}(\log\,E)$.
By \cite[Theorem 1.4]{Lac} and Lemma \ref{every lift=some lift} (1), it follows that $(Y, E)$ lifts to $W(k)$. 
Now, since $\pi^{*}D_W$ is a nef and big $\Q$-divisor whose support of the fractional part is contained in $E$, Theorem \ref{Hara's vanishing} shows that $0\neq s\in H^0(Y, \Omega_Y(\log\,E))\otimes \sO_Y(-\lceil \pi^{*}D_W \rceil)=0$, a contradiction.

\textbf{Step~2.}
Next, we assume that $\kappa(X, K_X+B)=0$ and prove the vanishing (b).
We can replace $(W, B_W)$ with the $(K_W+B_W)$-minimal model by Lemma \ref{push} and hence assume that $K_W+B_W\equiv 0$. 

\textbf{Step~2-1.}
First, we assume that $B_W\neq 0$ and $p>5$. In this case, $K_W$ is not pseudo-effective and we can run a $K_W$-MMP to obtain a birational contraction $\phi\colon W\to W'$ and a $K_{W'}$-Mori fiber space $f\colon W'\to Z$. Since $K_W+B_W\equiv 0$, the negativity lemma shows that $K_W+B_W=\phi^{*}(K_{W'}+B_{W'})$, where $B_{W'}\coloneqq \phi_{*}B_W$. Thus $(W', B_{W'})$ is log Calabi-Yau and $W'$ is klt.
By Lemma \ref{push}, we can replace $(W, B_W)$ with $(W', B_{W'})$. 
If $\dim\,Z=1$, then the assertion follows from Lemma \ref{LCTF}.
Thus we may assume that $\dim\,Z=0$. In this case, $W$ is a klt del Pezzo surface of Picard rank one and $D_W$ is an ample $\Q$-Cartier $\Z$-divisor.
Let $\pi\colon Y\to W$ be a log resolution of $(W, B_W)$, $B'\coloneqq \pi^{-1}_{*}B_W$, $E\coloneqq \Exc(\pi)$, and $B_Y\coloneqq B'+E$.
As in Step~1, we derive a contradiction assuming there exists an injective $\sO_Y$-module homomorphism $s \colon \sO_Y(\lceil \pi^{*}D_W\rceil) \hookrightarrow \Omega_Y(\log\,B_Y)$. 
Since $B'\neq 0$, we can take an irreducible component $B'_1$ of $B'$.
Since $B'_1$ is not $\pi$-exceptional and $D_W$ is an ample $\Q$-Cartier $\Z$-divisor,
an argument as in Step~1 shows that the morphism $s$ factors through $\sO_Y(\lceil \pi^{*}D_W \rceil) \to \Omega_Y(\log\, B_Y-B'_1)$. Since $K_W+B_W\equiv 0$, we have $\kappa(Y, K_Y+B_Y-B'_1)=-\infty$.
Now, we obtain a contradiction by Step~1.

\textbf{Step~2-2.}
Next, we assume that $B_W=0$. 
In this case, $W$ is a klt Calabi-Yau surface.
We take a positive integer $n$ as in Lemma \ref{max Gor index} and assume $p>n$.
We show that we may assume that $D_W$ is nef and big.
Let $D_W\equiv P+N$ be the Zariski decomposition. Note that we can take the Zariski decomposition even when $X$ is singular (\cite[Theorem 3.1]{Eno}).
We take a rational number $0<\epsilon \ll 1$ such that $(W, \epsilon N)$ is klt.
Since $K_W$ is torsion by the abundance theorem (\cite[Theorem 1.2]{Tan12}) and $N$ is negative definite, it follows that $\kappa(K_W+\epsilon N)=\kappa(X, N)=0$.
We run a $(K_W+\epsilon N)$-MMP to obtain a birational contraction $\phi\colon W\to W'$ to a $(K_W+\epsilon N)$-minimal model $W'$.
Then $K_{W'}=\phi_{*}K_W\equiv0$, and in particular, $W'$ is klt Calabi-Yau.
Moreover, $\phi_{*}\epsilon N\equiv K_{W'}+\phi_{*}\epsilon N\equiv 0$, and hence $\phi_{*}D_W\equiv\phi_{*}P$ is nef and big. By Lemma \ref{push}, we can replace $W$ with $W'$ and assume that $D_W$ is nef and big.

We next reduce to the case where $W$ is canonical Calabi-Yau.
We assume that $W$ is a strictly klt Calabi-Yau surface. 
By Lemma \ref{Cartier index}, the positive integer $n$ is the minimum integer such that $nK_W=0$.
Then we can take a cyclic cover $\tau\colon \tilde{W}\to W$ associated to a non-zero global section of $nK_W=0$.
Since $n$ is not divisible by $p$, it follows that $\tau$ is \'etale in codimension one, and hence we obtain an injective morphism
\[
H^0(W, (\Omega_X^{[1]}\otimes \sO_W(-D_W))^{**})\hookrightarrow H^0(\tilde{W}, (\Omega_{\tilde{W}}^{[1]}\otimes \sO_{\tilde{W}}(-\tau^{*}D_W))^{**})
\]
and $\tau^{*}D_W$ is nef and big.
By replacing $W$ with $\tilde{W}$, we may assume that $W$ has only canonical singularities.

Now, we show the vanishing (b). Let $\pi\colon Y\to W$ be the minimal resolution and $E\coloneqq \Exc(\pi)$. 
By Proposition \ref{ext thm}, it suffices to show that 
\[
H^0(Y, \Omega_{Y}(\log\,E)\otimes \sO_{Y}(-\lceil \pi^{*}D_W \rceil))=0.
\]
Since $p>n\geq 19$ by Remark \ref{Rem:max Gor index}, the pair $(Y, E)$ lifts to $W_2(k)$ by Proposition \ref{Prop : Lift of canonical CY}.
Thus we conclude the desired vanishing by Theorem \ref{Hara's vanishing}.

\textbf{Step~3.}
Finally, we assume that $\kappa(X, K_X+B)=1$ and $p>3$. We prove the vanishing (a) directly.
In this case, by replacing $(X, B)$ with its $(K_X+B)$-minimal model, we may assume that $K_X+B$ is semiample and $\kappa(X, K_X+B)=1$. Then there exists a projective morphism $f\colon X\to Z$ such that $\dim\,Z=1$, $f_{*}\sO_X=\sO_Z$, and $K_X+B$ is numerically trivial over $Z$. Now, by Lemma \ref{LCTF}, we obtain the assertion.

We will check the sharpness of the explicit bounds on $p_0$ in Example \ref{sharpness of char}.
\end{proof}

We recall the definition of a \textit{globally sharply $F$-split pair}, which is a positive characteristic analog of a log Calabi-Yau pair in characteristic zero.

\begin{defn}[\textup{\cite[Definition 3.1]{SS10}}]\label{defition:GFS} 
Let $(X, B)$ be a pair over an algebraically closed field of characteristic $p>0$. 
We say that $(X, B)$ is \textit{globally sharply $F$-split} if there exists a positive integer $e\in \Z_{>0}$ such that the composite map 
\[
\mathcal{O}_X \to F^e_*\mathcal{O}_X \hookrightarrow F^e_*\mathcal{O}_X(\lceil (p^e-1)B\rceil)
\]
of the $e$-times iterated Frobenius morphism $\mathcal{O}_X \to F^e_*\mathcal{O}_X$ and the natural inclusion $F^e_*\mathcal{O}_X \hookrightarrow F^e_*\mathcal{O}_X(\lceil (p^e-1)B\rceil)$ 
splits as an $\mathcal{O}_X$-module homomorphism. 
\end{defn}

By a similar argument to Theorem \ref{BSV, Intro}, we can show the Bogomolov-Sommese vanishing theorem for a globally sharply $F$-split surface pair.

\begin{prop}\label{BSV for GSFS}
Let $(X, B)$ be a globally sharply $F$-split surface pair over an algebraically closed field of characteristic $p>5$.
Then
\[
H^0(X, (\Omega_X^{[i]}(\log\,\lfloor B\rfloor)\otimes \sO_X(-D))^{**})=0
\]
for every $\Z$-divisor $D$ on $X$ satisfying $\kappa(X, D)>i$.
\end{prop}
\begin{proof}
By \cite[Theorem 4.4 (ii) and Theorem 4.3 (ii)]{SS10}, it follows that $(X, B)$ is lc and $-(K_X+B)$ is effective.
If $\kappa(X, K_X+\lfloor B\rfloor)=-\infty$, then the assertion follows from Theorem \ref{BSV, Intro}.

Thus we may assume that $K_X+\lfloor B\rfloor\equiv 0$.
First, we assume that $(X, \lfloor B\rfloor)$ is not klt.
By Proposition \ref{ext thm}, we can replace $(X, \lfloor B\rfloor)$ with its dlt blow-up. In this case, the boundary of the dlt pair is non-zero since $(X, \lfloor B\rfloor)$ is not klt. Then the assertion follows from Step~2-1 of the proof of Theorem \ref{BSV, Intro}.
 
Now, we assume that $X$ is klt Calabi-Yau and $B=0$. 
As in Step~2-2 of the proof of Theorem \ref{BSV, Intro}, by considering the Zariski decomposition, we can assume that $D$ is nef and big. Note that the globally $F$-split property is preserved under a birational contraction (\cite[1.1.9 Lemma]{fbook}).
Next, a splitting morphism $F_{*}\sO_X\to \sO_X$ give a non-zero section of $\Hom_{\sO_X}(F_{*}\sO_X, \sO_X)\cong H^0(X, \sO_X((1-p)K_X))$, and together with $K_X\equiv0$, we obtain $(1-p)K_X=0$.
In particular, the minimum positive integer $n$ such that $nK_X=0$ is not divisible by $p$.
We recall the globally $F$-split property is preserved under a finite cover which is \'etale in codimension one (\cite[Lemma 11.1.]{PZ19}).
Thus, by taking a cyclic cover associated to a non-zero global section of $nK_X$, we may assume that $X$ is a canonical Calabi-Yau surface such that $K_X=0$. If $X$ is an abelian surface, then the same argument as Step~2-2 of the proof of Theorem \ref{BSV, Intro} works.
Thus we may assume that the minimal resolution $Y$ of $X$ is a K3 surface. Now, by \cite[1.3.13 Lemma]{fbook} and \cite[5.1 Theorem]{GK}, the K3 surface $Y$ is not supersingular, and an argument of Step~2-2 of the proof of Theorem \ref{BSV, Intro} and Remark \ref{Rem : Lift of canonical CY} show the desired vanishing.
\end{proof}

\begin{rem}
In the proof of Proposition \ref{BSV for GSFS}, we do not need the assumption that $p>5$ in some places (see \cite[Lemma 4.17]{BBKW} and \cite[Proposition 4.18]{BBKW}). 
On the other hand, it is still open whether Proposition \ref{ext thm} holds for $F$-pure surface singularities in characteristic $p\leq5$. This is the reason why we need the assumption that $p>5$ in Proposition \ref{BSV for GSFS}.
\end{rem}

\section{Liftability of surface pairs}\label{sec : Lift of a surface pair}
In this section, we prove Theorems \ref{lift, Intro} and \ref{KVV, Intro}. We also discuss deformations of an lc projective surface whose canonical divisor has negative Iitaka dimension (Proposition \ref{tangent}). 
First, we focus on the vanishing of the second cohomology of the logarithmic tangent sheaf.

\begin{defn}\label{Definition:Q-ableian}
Let $X$ be a normal projective variety.
We say $X$ is \textit{Q-abelian} if there exists a finite surjective morphism $\tau\colon \tilde{X}\to X$ such that $\tilde{X}$ is an abelian variety and $\tau$ is \'etale in codimension one.
\end{defn}

\begin{prop}\label{van of tan}
Let $(X, B)$ be an lc projective surface pair over an algebraically closed field of characteristic $p>0$ such that $B$ is reduced.
When $\kappa(X, K_X+B)=0$, let $(X', B')$ be the $(K_X+B)$-minimal model of $(X, B)$, where $B'$ is the pushforward of $B$.
Suppose that one of the following holds.
\begin{enumerate}
    \item[\textup{(1)}] $\kappa(X, K_X+B)=-\infty$ and $p>5$.
    \item[\textup{(2)}] $\kappa(X, K_X+B)=0$ and one of the following holds.
    \begin{enumerate}
       \item[\textup{(i)}] $B'\neq 0$ and $p>5$,
       \item[\textup{(ii)}] $B'=0$, $X'$ is klt, the Gorenstein index of $X'$ is not divisible by $p$, $X'$ is not Q-abelian, and $p>19$.
    \end{enumerate}
\end{enumerate}
Then $H^2(X, T_X(-\log\,B))=0$.
\end{prop}
\begin{rem}
All the assumptions on $p$ are sharp (see Examples \ref{sharpness of char}, \ref{Example : num triv pair}, and \ref{Example : sharpness for Kodaira dim 0}).
\end{rem}
\begin{proof}
First, we assume that the condition (1) holds.
We can reduce the desired vanishing to an output of a $(K_X+B)$-MMP by Remark \ref{Remark:tangent}, and hence assume that $X$ admits a $(K_X+B)$-Mori fiber space structure $f\colon X\to Z$.
If $\dim\,Z=1$, then the assertion follows from Lemma \ref{LCTF} since $-K_X$ is $f$-ample.
We next assume that $\dim\,Z=0$. Note that $-K_X$ is $\Q$-Cartier by \cite[Theorem 5.4]{Tan12}. Then it follows from $\rho(X)=1$ that $-K_{X}$ is an ample $\Q$-Cartier $\Z$-divisor, and the assertion follows from Theorem \ref{BSV, Intro}.

Next, we assume that the condition (2)-(i) holds.
It suffices to show that $H^2(X', T_{X'}(-\log\,B'))=0$.
Since $K_{X'}+B'\equiv 0$ and $B'\neq 0$, it follows that $K_{X'}$ is not pseudo-effective. Then we can run a $K_{X'}$-MMP to obtain a birational contraction $\phi\colon X'\to \overline{X}$ to a $K_{\overline{X}}$-Mori fiber space $f\colon \overline{X}\to Z$. 
It suffices to show that $H^2(\overline{X}, T_{\overline{X}}(-\log\,\overline{B}))=0$.
Since $K_{X'}+B'\equiv 0$, the negativity lemma shows that $K_{X'}+B'=\phi^{*}(K_{\overline{X}}+\overline{B})$ and hence $(\overline{X}, \overline{B})$ is log Calabi-Yau, where $\overline{B}\coloneqq \phi_{*}B'$.
If $\dim\,Z=1$, then the assertion follows from Lemma \ref{LCTF} since $-K_{\overline{X}}$ is $f$-ample. If $\dim\,Z=0$, then the assertion follows from Step~2-1 of the proof of Theorem \ref{BSV, Intro} since $-K_{\overline{X}}$ is ample $\Q$-Cartier by \cite[Theorem 5.4]{Tan12}.

Finally, we assume that the condition (2)-(ii) holds.
It suffices to show that $H^2(X', T_{X'})=0$.
We first assume that $X'$ is strictly klt Calabi-Yau. 
Let $n$ be the minimum positive integer such that $nK_{X'}=0$.
Then $n$ is equal to the Gorenstein index by Lemma \ref{Cartier index}, and hence $n$ is not divisible by $p$ by assumption. 
Let $\tau\colon \tilde{X}\to X'$ be a cyclic cover associated to a non-zero global section of $nK_{X'}=0$.
Since $\tau$ is \'etale in codimension one, we have
\begin{align*}
H^2(X', T_{X'})\cong H^0(X', (\Omega_{X'}^{[1]}\otimes \sO_{X'}(K_{X'}))^{**})\hookrightarrow& H^0(\tilde{X}, (\Omega_{\tilde{X}}^{[1]}\otimes \sO_{\tilde{X}}(K_{\tilde{X}}))^{**})\\
\cong&H^2(\tilde{X}, T_{\tilde{X}}),
\end{align*}
Thus we may assume that $X'$ is canonical Calabi-Yau.
By the assumption that $X'$ is not Q-abelian, the minimal resolution of $X'$ is a K3 surface or an Enriques surface. 
In these cases, we have already shown that $H^2(X', T_{X'})=0$ in the proof of Proposition \ref{Rem : Lift of canonical CY}.
\end{proof}

Now, we prove Theorems \ref{lift, Intro} and \ref{KVV, Intro}.

\begin{proof}[Proof of Theorem \ref{lift, Intro}]
Set $B_Y\coloneqq \pi_{*}^{-1}B+\Exc(\pi)$.
Suppose that the condition (1) holds and $p>5$.
Then $\kappa(Y,K_Y+B_Y)=-\infty$ by Lemma \ref{push} (2), and hence $(Y, B_Y)$ lifts to $W(k)$ by Proposition \ref{van of tan} (1) and Theorem \ref{log smooth lift criterion}.

Next, we assume that the condition (2) holds and $p>5$.
By Lemma \ref{Lemma:dlt blow-up}, we can decompose $\pi\colon Y\to X$ into a birational morphism $Y\to W$ and a dlt blow-up $h\colon W\to X$. Then there exists an effective $\Q$-divisor $F$ such that $K_W+B_W+F\equiv h^{*}(K_X+B)\equiv 0$, where $B_W\coloneqq h^{-1}_{*}B+\Exc(h)$.
By assumption, we have $B_W\neq 0$ and hence $H^2(Y, T_Y(-\log\,B_Y))\hookrightarrow H^2(W, T_W(-\log\,B_W))=0$ by Proposition \ref{van of tan} (1) and (2)-(i). 
Moreover, since $-K_X\equiv B$ is strictly effective, it follows that $H^2(Y, \sO_Y)=0$. Now, we conclude that $(Y, B_Y)$ lifts to $W(k)$ by Theorem \ref{log smooth lift criterion}.

Finally, we assume that the condition (3) holds. 
In this case, $\kappa(Y,K_Y+B_Y)\leq 0$ by Lemma \ref{push} (2).
If $\kappa(Y,K_Y+B_Y)=-\infty$ and $p>5$, then $(Y, B_Y)$ lifts to $W(k)$ by (1).
Thus we can assume that $\kappa(Y,K_Y+B_Y)=0$.
By Propositions \ref{Prop : Lift of canonical CY} and \ref{Lift of strictly klt CY}, we can take a positive integer $p_{0}>19$ with the following property; for every klt Calabi-Yau surface $Z$ over an algebraically closed field of characteristic bigger than $p_0$ and every log resolution $f\colon \tilde{Z}\to Z$, the pair $(\tilde{Z}, \Exc(f))$ lifts to $W(k)$. We fix such a $p_0$ and assume that $p>p_0$.
We run a $(K_Y+B_Y)$-MMP to obtain a birational contraction $\phi\colon Y\to Y'$ to the $(K_Y+B_Y)$-minimal model $(Y',B_{Y'}\coloneqq \phi_{*}B_Y)$, which is dlt and log Calabi-Yau.
If $B_{Y'}\neq0$, then $(Y, B_Y)$ lifts to $W(k)$ by (2).
If $B_{Y'}=0$, then we obtain the desired liftability by the assumption of $p_0$.

We will check the sharpness of the explicit bound on $p_0$ in Examples \ref{sharpness of char} and \ref{Example : num triv pair}.
\end{proof}

\begin{proof}[Proof of Theorem \ref{KVV, Intro}]
If $\kappa(X, K_X)=-\infty$ and $p>5$, then we obtain the desired vanishing by Theorem \ref{lift, Intro} and Lemma \ref{Lem:KVV}.  
Similarly, if we take a positive integer $p_0$ as in Theorem \ref{lift, Intro}, $\kappa(X, K_X)=0$, and $p>p_0$, then we obtain the remaining case of (1).
Now, we assume that $\kappa(X, K_X)=1$, $X$ is lc, and $p>3$.
In this case, we have $H^0(X, (\Omega_X^{[1]}\otimes\sO_X(-p^eD))^{**})=0$ for all $e\in \Z_{>0}$ by Theorem \ref{BSV, Intro}.
Then, by the proof of \cite[Lemma 2.5]{Kaw2}, we have an injective morphism $H^1(X, \sO_X(-D))\hookrightarrow H^1(X, \sO_X(-p^{e}D))$ arising from the $e$-th iterated Frobenius morphism. 
Let $\pi\colon Y\to X$ be a log resolution.
By the proof of Lemma \ref{Lem:KVV}, it suffices to show that $H^1(Y, \sO_{Y}(-\lceil p^e\pi^{*}D\rceil))=0$ for $e\gg0$.
We take $m, n\in \Z_{>0}$ such that $p^m(p^n-1)\pi^{*}D$ is Cartier.
Then we obtain 
\[
H^1(Y, \sO_Y(-\lceil p^{m+nl}\pi^{*}D \rceil)\\
=H^1(Y, \sO_Y(-\lceil p^m\pi^{*}D \rceil+(\sum_{i=0}^{l-1}p^{ni})p^m(p^n-1)\pi^{*}D))
\]
and the last cohomology vanishes for $l\gg0$ by \cite[Theorem 2.6]{Tan15}.

We will check the sharpness of the explicit bounds on $p_0$ in Example \ref{sharpness of char}.
\end{proof}

Finally, we apply Proposition \ref{van of tan} to show the vanishing of local-to-global obstruction (Definition \ref{def:local-to-global}).
In particular, we prove Proposition \ref{tangent} (3), which is a positive characteristic analog of \cite[Proposition 3.1]{HP}.

We first recall the definition of local-to-global obstruction, whose vanishing means deformations of singular points extend to a global deformation.

\begin{defn}\label{def:local-to-global}
Let $X$ be a normal projective variety over an algebraically closed field $k$ with only isolated singularities.
Let $\Lambda\coloneqq k$ or a complete discrete valuation ring with residue field $k$.
Let $T$ be a Noetherian scheme of finite type over $\Lambda$ and $o \colon \Spec k \to T$ a morphism.
For every singular point $x\in X$, we suppose that there exists an \'etale morphism $\phi_{x}\colon U_x\to X$ and a closed point $x'\in U_x$ such that 
\begin{itemize}
    \item $\phi_{x}(x')=x$
    \item $U_x$ is smooth outside $x'$, and
    \item there exists a flat morphism $\mathcal{U}_x\to T$ such that the base change of $\mathcal{U}_x$ by $o \colon \Spec k \to T$ is isomorphic to $U_x$.
\end{itemize}
We say $X$ has \textit{no local-to-global obstruction} if there exist
\begin{enumerate}
    \item an \'etale morphism $\iota\colon T'\to T$ from a Noetherian scheme $T'$ with a morphism $o' \colon\Spec\,k\to T'$ such that $\iota\circ o'=o$ and 
    \item there exists a flat projective morphism $\mathcal{X}\to T'$ such that the base change of $\mathcal{X}$ by $o'\colon\Spec\,k\to T'$ is isomorphic to $X$,
\end{enumerate}
such that the formal completions of $\sO_{\mathcal{X},x}$ and $\sO_{\mathcal{U}_x,x'}$ are isomorphic for every singular point $x\in X$.
\end{defn}

\begin{thm}\label{Thm:local-to-global}
Let $X$ be a normal projective variety over an algebraically closed field with only isolated singularities.
Suppose that $H^2(X, T_X)=H^2(X, \sO_X)=0$.
Then $X$ has no local-to-global obstruction.
\end{thm}
\begin{proof}
This is \cite[Theorem 4.14 and Remark 4.15]{LN}.
\end{proof}

\begin{defn}
Let $X$ be a normal projective variety. We say $X$ admits a \textit{$\Q$-Gorenstein smoothing} if there exists a flat projective morphism $\mathcal{X}\to T$ from a normal $\Q$-Gorenstein scheme to a smooth curve $T$ with a closed point $o\in T$ such that the fiber over $o$ is isomorphic to $X$ and $\mathcal{X}\to T$ is smooth over $T\setminus \{o\}$.
\end{defn}

\begin{defn}
Let $R$ be an $F$-finite ring of positive characteristic.
We say that $R$ is \textit{$F$-pure} if the Frobenius map $F\colon R\to F_{*}R$ splits as an $R$-module homomorphism.
We say that a variety $X$ is \textit{$F$-pure} if $\sO_{X,x}$ is $F$-pure for every closed point $x\in X$.
\end{defn}

\begin{prop}\label{tangent}
Let $X$ be an lc projective surface over an algebraically closed field $k$ of characteristic $p>5$ with $\kappa(X, K_X)=-\infty$.
Then $X$ has no local-to-global obstruction.
In particular, the following hold.
\begin{enumerate}
    \item[\textup{(1)}] If $X$ is $F$-pure, then $X$ lifts to $W_2(k)$.
    \item[\textup{(2)}] If $X$ is l.c.i, then $X$ lifts to $W(k)$.  
    \item[\textup{(3)}] If $X$ has only rational double points or toric singularities of class $T$ (see \cite[Definition 3.4]{LN} for the definition), then $X$ admits a $\Q$-Gorenstein smoothing.
\end{enumerate}
\end{prop}
\begin{proof}
By taking $B=0$ in Proposition \ref{van of tan} (1), we have $H^2(X, T_X)=0$.
Together with $H^2(X, \sO_X)=0$, it follows from Theorem \ref{Thm:local-to-global} that $X$ has no local-to-global obstruction.

First, we show (1). By \cite[Corollary 8]{Lan15}, a spectrum of an $F$-pure ring lifts to $W_2(k)$. Since $H^2(X, T_X)=0$, it follows that $X$ lifts to $W_2(k)$ by \cite[Theorem 4.13]{LN}.

Next, we show (2). By applying \cite[Theorem 9.2]{Har2} repeatedly, we can see that an l.c.i. affine variety lifts to $W(k)$. Then it follows that $X$ lifts to a scheme $T'$ \'etale over $W(k)$ since $X$ has no local-to-global obstruction. 
Since $T'$ is smooth over $\Z$, we can conclude that $X$ lifts to $W(k)$ by the proof of \cite[Proposition 2.5]{ABL}.

Finally, (3) follows from \cite[Theorem 5.3]{LN}.
\end{proof}

\section{Sharpness of Theorems \ref{BSV, Intro}, \ref{lift, Intro}, and \ref{KVV, Intro}}\label{sec:examples}

In this section, we observe the failure of Theorems \ref{BSV, Intro}, \ref{lift, Intro}, and \ref{KVV, Intro} in low characteristic or for a surface pair whose log canonical divisor is big.
First, we focus on the characteristic.

\begin{eg}\label{sharpness of char}
By \cite[Theorem 1.7 (3)]{KN}, \cite[Theorem 1.1]{Ber}, and \cite[Proposition 5.2]{ABL}, we can take a klt del Pezzo surface $X$ in each characteristic $p\in \{2,3,5\}$ with more than four singularities and an ample $\Q$-Cartier $\Z$-divisor $D$ on $X$ such that $H^1(X, \sO_X(K_X+D))\neq 0$.
Let $\pi\colon Y\to X$ be the minimal resolution with $E\coloneqq \Exc(\pi)$. Then $-K_Y$ is big and $\kappa(Y, K_Y+E)=-\infty$. 

Firstly, $(Y, E)$ does not lift to any Noetherian local domain with fractional field of characteristic zero because there are no klt del Pezzo surfaces with more than four singularities in characteristic zero by \cite[Theorem 1.1]{Bel}.
In addition, $(Y, E)$ dose not lift to $W_2(k)$ by Lemma \ref{Lem:KVV}, and it follows from Theorem \ref{log smooth lift criterion} that $0\neq H^2(Y, T_Y(-\log\,E))\hookrightarrow H^2(X, T_X)$.
Therefore, the explicit bound $p_0=5$ in Theorems \ref{BSV, Intro}, \ref{lift, Intro} (1), and \ref{KVV, Intro} (1) is optimal.
These examples also show that the sharpness of the assumption $p>5$ in Proposition \ref{van of tan} (1).

By \cite[Section 3,1]{Ray}, we can take a smooth projective surface $X$ in each characteristic $p\in \{2,3\}$ with $\kappa(X, K_X)=1$ and an ample Cartier divisor $D$ on $X$ such that $H^1(X, \sO_X(K_X+D))\neq 0$.
Then \cite[Lemma 2.5]{Kaw2} shows that there exists $n>0$ such that $H^0(X, \Omega_X\otimes \sO_X(-p^nD))\neq 0$. Therefore, the explicit bound $p_0=3$ in Theorems \ref{BSV, Intro}, \ref{KVV, Intro} (2) and the assumption that $p>3$ in Lemma \ref{LCTF} are optimal.
\end{eg}

\begin{eg}\label{Example : num triv pair}
We show that there exists a klt del Pezzo surface $X$ in each characteristic $p\in\{2,3,5\}$ and a non-zero reduced divisor $B$ on $X$ such that $K_X+B\equiv 0$, but $(Y, f_{*}^{-1}B+\Exc(f))$ does not lift to any Noetherian local domain with fractional field of characteristic zero for some log resolution $f\colon Y\to X$ of $(X, B)$.

We first assume $p=5$. We take a del Pezzo surface $X$ with two $A_4$-singularities and a cuspidal rational curve $B$ in the smooth locus of $X$ as in \cite[Example 7.6]{Lac}.
Then we have $K_X+B\equiv 0$ by the adjunction formula. 
We take a three-times blow-up $f\colon Y\to X$ at the cusp of $B$.
Then there exists a contraction $\pi\colon Y\to Z$ to a klt del Pezzo surface $Z$ with five singularities and $\Exc(\pi)\subset f_{*}^{-1}B+\Exc(f)$ (see \cite[Example 7.6]{Lac} for the detail). Then $(Y, \Exc(\pi))$ does not lift to any Noetherian local domain with fractional field of characteristic zero by \cite[Theorem 1.1]{Bel} and neither does $(Y, f_{*}^{-1}B+\Exc(f))$.
When $p=3$, we can take $X=\PP_k^2$ and a curve $B$ as in \cite[Example 7.5]{Lac} to show the assertion.
When $p=2$, we can take a del Pezzo surface $X$ as any one of \cite[Theorem 1.7 (2)]{KN} and $B$ is a general anti-canonical member, which is integral. 

Therefore, the explicit bound on $p_0$ in Theorem \ref{lift, Intro} (2) and the assumption $p>5$ in Proposition \ref{van of tan} (2)-(i) are optimal.
\end{eg}

The following example was taught by Fabio Bernasconi, Iacopo Brivio, and Jakub Witaszek.

\begin{eg}\label{Example : sharpness for Kodaira dim 0}
By \cite[Corollary 1.2]{Shi}, there exists a canonical Calabi-Yau surface $X$ in each characteristic $p\leq 19$ such that $Y$ is a supersingular K3 surface and $E\coloneqq \Exc(\pi)$ consists of $21$ $(-2)$-curves, where $\pi\colon Y\to X$ is the minimal resolution.
Then $(Y, E)$ does not lift to any Noetherian local domain $R$ with fractional field $K$ of characteristic zero. For the sake of contradiction, we assume that there exists a lifting $(\mathcal{Y}, \mathcal{E})$ of $(Y,E)$ to such an $R$. 
Let $Y_{\overline{K}}$ and $E_{\overline{K}}$ be the geometric generic fibers of $\mathcal{Y}\to R$ and $\mathcal{E}\to R$, respectively.

We show that $Y_{\overline{K}}$ is a K3 surface.
Since $H^1(Y, \sO_Y)=0$, a lifting of each invertible sheaf is unique by \cite[Corollary 8.5.5]{FAG}. Then $\omega_{\mathcal{Y}}|_Y=\omega_Y=\sO_Y=\sO_{\mathcal{Y}}|_Y$ shows that $\omega_{\mathcal{Y}}=\sO_{\mathcal{Y}}$.
Together with $\mathcal{X}(Y_{\overline{K}}, \sO_{Y_{\overline{K}}})=\mathcal{X}(Y, \sO_Y)=2$, we conclude that $Y_{\overline{K}}$ is a K3 surface.
Since $Y_{\overline{K}}$ contains $21$ $(-2)$-curves $E_{\overline{K}}$ that is negative definite, we obtain $\rho(Y_{\overline{K}})\geq 22$, a contradiction with the fact that the Picard rank of a K3 surface in characteristic zero is at most $20$ (see \cite[Chapter 17, 1.1]{K3book} for example).
Finally, by the proof Proposition \ref{Prop : Lift of canonical CY}, we obtain $0\neq H^2(Y, T_Y(-\log\,E))\hookrightarrow H^2(X, T_X)$. 

Therefore, $p_0$ in Theorem \ref{lift, Intro} (3) should be at least $19$. Moreover, the assumption that $p>19$ in Propositions \ref{Prop : Lift of canonical CY} and \ref{van of tan} (2)-(ii) is sharp.
\end{eg}

Finally, we close the paper by discussing the assumptions of Iitaka dimensions of log canonical divisors in Theorems \ref{BSV, Intro}, \ref{lift, Intro}, and \ref{KVV, Intro}. 
By Raynaud's counterexample (\cite{Ray}) to the Kodaira vanishing theorem on a smooth projective surface with a big canonical divisor, we can see that Theorems \ref{BSV, Intro}, \ref{lift, Intro}, and \ref{KVV, Intro} do not hold for a surface with a big canonical divisor in any characteristic.
In the next example, we will see that Langer's surface pair \cite{Lan15} shows that Theorems \ref{BSV, Intro} and \ref{lift, Intro} do not hold on a surface pair whose log canonical divisor is big even when the surface itself is rational.

\begin{eg}\label{Example:Langer's surface}
We first recall the construction of Langer's surface pair \cite[Section 8]{Lan15}.
Let $h\colon X\to \PP_k^2$ be the blow-up all the $\mathbb{F}_p$-rational points and $L_1,\ldots, L_{p^2+p+1}$ strict transforms of all the $\mathbb{F}_p$-lines.
Then $X$ is a smooth rational surface and $L_1,\ldots,L_{p^2+p+1}$ are pairwise disjoint smooth rational curves. 

By \cite[Theorem 3.1]{CT}, there exists a nef and big $\Q$-divisor $D$ such that $H^1(X, \sO_X(K_X+\lceil D \rceil))\neq 0$ and $\Supp(\{D\})=\sum_{i=1}^{p^2+p+1}L_i$.
Thus $(X, \sum_{i=1}^{p^2+p+1}L_i)$ dose not lift to $W_2(k)$ by Theorem \ref{Hara's vanishing} and $H^2(X, T_X(-\log \sum_{i=1}^{p^2+p+1}L_i))\neq 0$ by Theorem \ref{log smooth lift criterion}.
Finally, there exists a big divisor $M$ such that $\sO_X(M)$ is contained in $\Omega_X(\log\,\sum_{i=1}^{p^2+p+1}L_i)$ by \cite[Proposition 11.1]{Langer19}.

Now, we check that $K_X+\sum_{i=1}^{p^2+p+1}L_i$ is big except when $p=2$.
Since $L_i^2=-p$ for each $i$ and $L_1,\ldots,L_{p^2+p+1}$ are pairwise disjoint, we can take the contraction $f\colon X\to Z$ of $\sum_{i=1}^{p^2+p+1}L_i$.
By the proof of \cite[Lemma 2.4 (i)]{CT},
we have $K_X+(1-\frac{2}{p})(\sum_{i=1}^{p^2+p+1}L_i)=f^{*}K_Z$ and hence
$K_X+\sum_{i=1}^{p^2+p+1}L_i=\lceil f^{*}K_Z \rceil$.
If $p\neq2$, then $K_Z$ is ample by \cite[Lemma 2.4 (iv)]{CT} and hence $K_X+\sum_{i=1}^{p^2+p+1}L_i$ is big.
Note that if $p=2$, then $\kappa(X, K_X+\sum_{i=1}^{p^2+p+1}L_i)=-\infty$ since $f_{*}(K_X+\sum_{i=1}^{p^2+p+1}L_i)=K_Z$ is anti-ample.

Therefore, Theorems \ref{BSV, Intro}, \ref{lift, Intro} and Proposition \ref{van of tan} do not hold on a surface pair whose log canonical divisor is big even when the surface itself is rational. 
\end{eg}

\begin{rem}\label{Remark:log lift}
For a singular surface, it is often more useful to consider the liftability of a log resolution than that of itself (see \cite{ABL}, \cite{CTW}, and \cite{KN} for example).
In Example \ref{Example:Langer's surface}, we constructed the pathological example from the log resolution of the pair consisting of $\PP_{k}^2$ and all the $\mathbb{F}_p$-lines $\sum_{i=1}^{p^2+p+1}\overline{L}_i$.
However, the pair $(\PP_{k}^2, \sum_{i=1}^{p^2+p+1}\overline{L}_i)$ clearly lifts to $W(k)$.
For this reason, when we discuss lifting of a non-log smooth pair, it is more suitable to consider the liftability of a log resolution of the pair to capture pathologies in positive characteristic.  
\end{rem}

\begin{rem}
Example \ref{Example:Langer's surface} gives a slightly generalization of \cite[Corollary 3.3]{CT}. Indeed, we can drop the assumption $p\geq 3$ and replace $\sum_{i=1}^{q^2+q+1}E_i+\sum_{i=1}^{q^2+q+1}L'_i$ with $\sum_{i=1}^{q^2+q+1}L'_i$ in \cite[Corollary 3.3]{CT}.
On the other hand, this fact also follows from \cite[Proposition 4.1]{Lan} and \cite[Proposition 11.1]{Langer19}. 
\end{rem}


\section*{Acknowledgements}
The author wishes to express his gratitude to Professor Shunsuke Takagi for his suggestions and advice.
The author is also grateful to Masaru Nagaoka for discussion about del Pezzo surfaces, Fabio Bernasconi for giving valuable comments that strengthen results in this paper, Shou Yoshikawa for telling him Remark \ref{Remark:lifting as log smooth pairs}, Teppei Takamatsu, Makoto Enokizono,
Iacopo Brivio, and Jakub Witaszek for helpful discussions.
He also would like to thank the referee for valuable suggestions that improved the paper.
This work was supported by JSPS KAKENHI Grant number JP19J21085.

\bibliography{hoge.bib}
\bibliographystyle{abbrv}

\end{document}